\documentclass[american,english]{article}
\usepackage[T1]{fontenc}
\usepackage[latin9]{inputenc}
\usepackage{amsmath}
\usepackage{amsthm}
\usepackage{amssymb}
\PassOptionsToPackage{normalem}{ulem}
\usepackage{ulem}

\makeatletter
\numberwithin{equation}{section}
\numberwithin{figure}{section}
\theoremstyle{plain}
\newtheorem{thm}{\protect\theoremname}[section]
\theoremstyle{definition}
\newtheorem{problem}[thm]{\protect\problemname}
\theoremstyle{plain}
\newtheorem{prop}[thm]{\protect\propositionname}
\theoremstyle{plain}
\newtheorem{cor}[thm]{\protect\corollaryname}
\theoremstyle{remark}
\newtheorem{rem}[thm]{\protect\remarkname}
\theoremstyle{plain}
\newtheorem{lem}[thm]{\protect\lemmaname}
\theoremstyle{definition}
\newtheorem{defn}[thm]{\protect\definitionname}

\@ifundefined{date}{}{\date{}}
\makeatother

\usepackage{babel}
\addto\captionsamerican{\renewcommand{\corollaryname}{Corollary}}
\addto\captionsamerican{\renewcommand{\definitionname}{Definition}}
\addto\captionsamerican{\renewcommand{\lemmaname}{Lemma}}
\addto\captionsamerican{\renewcommand{\problemname}{Problem}}
\addto\captionsamerican{\renewcommand{\propositionname}{Proposition}}
\addto\captionsamerican{\renewcommand{\remarkname}{Remark}}
\addto\captionsamerican{\renewcommand{\theoremname}{Theorem}}
\addto\captionsenglish{\renewcommand{\corollaryname}{Corollary}}
\addto\captionsenglish{\renewcommand{\definitionname}{Definition}}
\addto\captionsenglish{\renewcommand{\lemmaname}{Lemma}}
\addto\captionsenglish{\renewcommand{\problemname}{Problem}}
\addto\captionsenglish{\renewcommand{\propositionname}{Proposition}}
\addto\captionsenglish{\renewcommand{\remarkname}{Remark}}
\addto\captionsenglish{\renewcommand{\theoremname}{Theorem}}
\providecommand{\corollaryname}{Corollary}
\providecommand{\definitionname}{Definition}
\providecommand{\lemmaname}{Lemma}
\providecommand{\problemname}{Problem}
\providecommand{\propositionname}{Proposition}
\providecommand{\remarkname}{Remark}
\providecommand{\theoremname}{Theorem}

\begin{document}
\title{On Pro-$2$ Identities of $2\times2$ Linear Groups}
\author{David El-Chai Ben-Ezra, Efim Zelmanov}
\maketitle
\begin{abstract}
Let $\hat{F}$ be a free pro-$p$ non-abelian group, and let $\Delta$
be a commutative Noetherian complete local ring with a maximal ideal
$I$ such that $\textrm{char}(\Delta/I)=p>0$. In \cite{key-3}, Zubkov
showed that when $p\neq2$, the pro-$p$ congruence subgroup
\[
GL_{2}^{1}(\Delta)=\ker(GL_{2}(\Delta)\overset{\Delta\to\Delta/I}{\longrightarrow}GL_{2}(\Delta/I))
\]
admits a pro-$p$ identity, i.e., there exists an element $1\neq w\in\hat{F}$
that vanishes under any continuous homomorphism $\hat{F}\to GL_{2}^{1}(\Delta)$.

In this paper we investigate the case $p=2$. The main result is that
when $\textrm{char}(\Delta)=2$, the pro-$2$ group $GL_{2}^{1}(\Delta)$
admits a pro-$2$ identity. This result was obtained by the use of
trace identities that originate in PI-theory.
\end{abstract}
\textbf{}

\tableofcontents{}

\section{Introduction}

It is well known that discrete free non-abelian groups are linear,
and can easily be embedded in the group $GL_{2}(\mathbb{Z})$. Surprisingly,
in the category of pro-$p$ groups, the problem of linearity is still
open. 

We say that $\Delta$ is a \uline{pro-\mbox{$p$} ring} if $\Delta$
is a commutative Noetherian complete local ring with a maximal ideal
$I$ such that $\Delta/I$ is a finite field of characteristic $p$.
In this case 
\[
\Delta=\underset{n}{\underleftarrow{\lim}}(\Delta/I^{n})
\]
is a profinite ring, and for any $d$, the congruence subgroup
\[
GL_{d}^{1}(\Delta)=\ker(GL_{d}(\Delta)\overset{\Delta\to\Delta/I}{\longrightarrow}GL_{d}(\Delta/I))
\]
is a pro-$p$ group. 
\begin{problem}
Let $\hat{F}$ be a non-abelian free pro-$p$ group. Can $\hat{F}$
be continuously embedded in $GL_{d}^{1}(\Delta)$ for some $d$ and
a pro-$p$ ring $\Delta$?
\end{problem}

In fact, the known partial results suggest that the answer should
be negative. For example, it is known that a free pro-$p$ non-abelian
group cannot be embedded as a closed subgroup in the pro-$p$ groups
(see \cite{key-4}, \cite{key-6})
\begin{align*}
GL_{d}^{1}(\mathbb{Z}_{p}) & =\ker(GL_{d}(\mathbb{Z}_{p})\overset{\mathbb{Z}_{p}\to\mathbb{F}_{p}}{\longrightarrow}GL_{d}(\mathbb{F}_{p}))\\
GL_{d}^{1}(\mathbb{F}_{p}\left\langle \left\langle t\right\rangle \right\rangle ) & =\ker(GL_{d}(\mathbb{F}_{p}\left\langle \left\langle t\right\rangle \right\rangle )\overset{t\mapsto0}{\longrightarrow}GL_{d}(\mathbb{F}_{p})).
\end{align*}

Let $\hat{F}$ be a free pro-$p$ group, and $\hat{H}$ a pro-$p$
group. We say that an element $1\neq w\in\hat{F}$ is a \uline{pro-\mbox{$p$}
identity} of $\hat{H}$ if $w$ vanishes under every (continuous)
homomorphism $\hat{F}\to\hat{H}$. In \cite{key-3}, using the idea
of generic matrices, Zubkov showed that given a fixed $d$, the following
conditions are equivalent \cite{key-3}:
\begin{itemize}
\item $\hat{F}$ cannot be embedded in $GL_{d}^{1}(\Delta)$ for some pro-$p$
ring $\Delta$.
\item There exists an element $1\neq w\in\hat{F}$ that serves as a pro-$p$
identity of every pro-$p$ group of the form $GL_{d}^{1}(\Delta)$,
where $\Delta$ is a pro-$p$ ring.
\end{itemize}
Then, Zubkov showed that these conditions are satisfied for $d=2$
whenever $p\neq2$ \cite{key-3}. In particular, for every $p\neq2$,
a free pro-$p$ group cannot be embedded, as a closed subgroup, in
$GL_{2}^{1}(\Delta)$, where $\Delta$ is a pro-$p$ ring. Later,
using ideas from the solution of the Specht problem, the second author
announced that given a fixed $d$, the aforementioned conditions are
satisfied for every large enough prime $d\ll p$ (see \cite{key-5},
\cite{key-13}). 

Given these results, it is natural to ask what happens when $p$ is
not large enough. More specifically, what happens in the case where
$d=p=2$? Investigating this case is the main purpose of this paper.
Here is the main result (see $\mathsection$\ref{sec:char=00003D2}):
\begin{thm}
\label{thm:main-main-1} Let $\Delta$ be a pro-$2$ ring of $\textrm{char}(\Delta)=2$.
Then, $GL_{2}^{1}(\Delta)$ admits a pro-$2$ identity that is independent
in $\Delta$.
\end{thm}

From Theorem \ref{thm:main-main-1} we get that a free pro-$2$ group
cannot be embedded in $GL_{2}^{1}(\Delta)$ when $\Delta$ is a pro-$2$
ring of $\textrm{char}(\Delta)=2$. Actually, one can derive from
here that a free pro-$2$ group cannot be embedded in $GL_{2}^{1}(\Delta)$
whenever $\textrm{char}(\Delta)=2^{m}$ for some $m$. The problem
of whether a free pro-$2$ group can be embedded in $GL_{2}^{1}(\Delta)$
when $\Delta$ is a pro-$2$ ring of $\textrm{char}(\Delta)=0$ is
still open.

The main idea of the proof of Theorem \ref{thm:main-main-1} is the
use of trace identities that originate in PI-theory (see \cite{key-7,key-8,key-9,key-10,key-11}),
combined with some basic ideas and tools from \cite{key-3} and the
theory of Noetherian rings, such as Hilbert's basis theorem and the
Artin-Rees lemma. 

In order to help the reader understand the proof of Theorem \ref{thm:main-main-1},
we begin the paper with a review of Zubkov's approach (see $\mathsection$\ref{sec:Zubkov}).
Toward the end of the paper we describe exactly where Zubkov's argument
fails when $d=p=2$ (see $\mathsection$\ref{sec:char=00003D0}).
This description allows us to show that when $d=2$, there is a dichotomy
between $p=2$ and $p\neq2$, and in some sense, $2\times2$ linear
pro-$2$ groups have less pro-$2$ identities (if any).

Throughout the paper we use the notation $[a,b]=ab-ba$ for the Lie-commutator
of $a$ and $b$, and $[a,b,c]=[[a,b],c]$. For the group commutator
in a group $H$ we will use the notation $[g,h]_{H}=ghg^{-1}h^{-1}$
and $[g,h,k]_{H}=[[g,h]_{H},k]_{H}$. Brackets of the form $\left\langle \,\,\right\rangle $
will denote generation in the discrete sense, and double brackets
$\left\langle \left\langle \,\,\right\rangle \right\rangle $ will
denote generation in the topological sense. Throughout the paper,
whenever $\hat{H}$ is a pro-$p$ group, and we use the notion ``derived
subgroup'' (resp. ``lower central series'') we mean ``the closure
of the derived subgroup'' (resp. ``the closure of the lower central
series'').

\textbf{Acknowledgments: }During the period of the research, the first
author was supported by NSF research training grant (RTG) \# 1502651.
We would also like to thank the referee for his/her wise remarks that
in particular have led to some substantial simplifications of the
proof of Theorem \ref{thm:main-main-1}. 

\section{\label{sec:Zubkov}Review of Zubkov's approach}

The main goal of this section is to help the reader understand some
basic concepts lying behind the proof of Theorem \ref{thm:main-main-1}.
The discussion here will also be needed as a background for the dicussion
in $\mathsection$\ref{sec:char=00003D0}.

\subsection{\label{subsec:universal 1-1}The Universal Representation}

\selectlanguage{american}%
Let $x_{i,j}$ and $y_{i,j}$ for $1\leq i,j\leq d$ be free commuting
variables, and let
\[
\Pi_{*}=\mathbb{Z}_{p}\left\langle \left\langle x_{i,j},y_{i,j}\,|\,1\leq i,j\leq d\right\rangle \right\rangle 
\]
be the associative ring (with identity) of formal power series in
$x_{i,j}$ and $y_{i,j}$ over the $p$-adic numbers $\mathbb{Z}_{p}$.
Every element in $\Pi_{*}$ can be written as $f=\sum_{i=0}^{\infty}f_{i}$
where $f_{i}$ is homogeneous of degree $i$ \footnote{By degree we mean that $\deg(\prod_{i,j=1}^{d}x_{i,j}^{\alpha_{i,j}}y_{i,j}^{\beta_{i,j}})=\sum_{i,j=1}^{d}(\alpha_{ij}+\beta_{ij})$.}.
Denote 
\[
\Pi_{*}\vartriangleright Q_{*n}=\left\{ f=\sum_{i=0}^{\infty}f_{i}\in\Pi_{*}\,|\,f_{0},...,f_{n-1}=0\,,\deg(f_{i})=i\right\} .
\]
The finite index ideals $B_{n,m}=Q_{*n}+p^{m}\Pi_{*}$ serve as a
basis of neighborhoods of zero for the profinite topology of $\Pi_{*}$,
making $\Pi_{*}$ a pro-$p$ ring, with a maximal ideal $B_{1,1}=Q_{*1}+p\Pi_{*}$.

Endowed with the topology that comes from the congruence ideals 
\[
M_{d}(\Pi_{*},B_{n,m})=\ker(M_{d}(\Pi_{*})\to M_{d}(\Pi_{*}/B_{n,m}))
\]
as a basis of neighborhoods of zero, $M_{d}(\Pi_{*})$ is a profinite
ring. One can see that this topology makes the group $1+M_{d}(\Pi_{*},Q_{*1})$
a pro-$p$ group.

Denote the \uline{generic matrices} $x_{*},y_{*}\in M_{d}(\Pi_{*},Q_{*1})$
by 
\[
x_{*}=(x_{ij})_{i,j=1}^{d},\,\,\,\,\,\,y_{*}=(y_{ij})_{i,j=1}^{d}.
\]
Let $\hat{F}=\left\langle \left\langle X,Y\right\rangle \right\rangle $
be the free pro-$p$ group generated by $X,Y$. By the above, there
is a natural (continuous) homomorphism $\pi_{*}:\hat{F}\to1+M_{d}(\Pi_{*},Q_{*n})$
defined by
\[
X\mapsto1+x_{*},\,\,\,\,\,\,Y\mapsto1+y_{*}.
\]
We denote by $G_{*}=\left\langle 1+x_{*},1+y_{*}\right\rangle \subseteq1+M_{d}(\Pi_{*},Q_{*1})$
the (discrete) subgroup generated by $1+x_{*}$ and $1+y_{*}$, and
$\hat{G}_{*}\subseteq1+M_{d}(\Pi_{*},Q_{*1})$ its closure in $1+M_{d}(\Pi_{*},Q_{*1})$.
Then, $\hat{G}_{*}$ is a pro-$p$ group. The map \foreignlanguage{english}{
\[
\pi_{*}:\hat{F}\twoheadrightarrow\hat{G}_{*}\subseteq1+M_{d}(\Lambda_{*},Q_{*1})
\]
is called \uline{the universal representation}. The following theorem
justifies the name of $\pi_{*}$ (see the proof of Theorem 2.1 in
\cite{key-3}):}
\selectlanguage{english}%
\begin{thm}
\label{thm:universal-1}Let $\Delta$ be any pro-$p$ ring. Then,
every 
\[
1\neq w(X,Y)\in\ker\pi_{*}
\]
 is a pro-$p$ identity of $GL_{d}^{1}(\Delta)$. 
\end{thm}

\subsection{\label{subsec:structure-1}The structure of a minimal component}

\selectlanguage{american}%
Define the following notation:
\selectlanguage{english}%
\begin{itemize}
\item $\mathring{\Pi}_{*}=\mathbb{Q}_{p}\left\langle \left\langle x_{i,j},y_{i,j}\,|\,1\leq i,j\leq d\right\rangle \right\rangle $
= the \foreignlanguage{american}{ring of power series in $x_{i,j}$
and $y_{i,j}$ for $1\leq i,j\leq d$ over $\mathbb{Q}_{p}$. }
\selectlanguage{american}%
\item $\mathring{Q}_{*n}=\left\{ f=\sum_{i=0}^{\infty}f_{i}\in\mathring{\Pi}_{*}\,|\,f_{0},...,f_{n-1}=0\,,\deg(f_{i})=i\right\} $.
\item $A_{*}=\mathbb{Z}_{p}\left\langle \left\langle x_{*},y_{*}\right\rangle \right\rangle \subseteq M_{2}(\Pi_{*})$
= the ring of power series in $x_{*},y_{*}$ over $\mathbb{Z}_{p}$.
\item $\mathring{L}_{*}$ = the Lie algebra generated by $x_{*},y_{*}$
over $\mathbb{Q}_{p}$. 
\item $\mathring{L}_{*}^{(n)}$ = the subspace of $\mathring{L}_{*}$ of
homogeneous elements of degree $n$.
\item $L_{*}\subseteq\mathring{L}_{*}$ = the Lie algebra generated by $x_{*},y_{*}$
over $\mathbb{Z}_{p}$.
\item $L_{*}^{(n)}\subseteq\mathring{L}_{*}^{(n)}$ = the additive subgroup
of $L_{*}$ of homogeneous elements of degree $n$.
\item For an element of the form $g=1+a_{n}+a_{n+1}+...\in1+M_{2}(\Pi_{*},Q_{*1})$
\foreignlanguage{english}{where $a_{i}$ is the term of $g$ of degree
$i$, and $a_{n}\neq0$, we denote
\[
\min(g)=a_{n}.
\]
}
\end{itemize}
The following proposition was proved in \cite{key-3}:
\begin{prop}
\label{prop:min-1}Let $1\neq g\in\hat{G}_{*}\subseteq1+M_{2}(\Pi_{*},Q_{*1})$.
Then, for some $n$ we have
\[
\min(g)\in\mathring{L}_{*}^{(n)}\cap A_{*}.
\]
\end{prop}

Here is an outline of the proof. It is obvious that $\min(g)\in A_{*}$.
To show $\min(g)\in\mathring{L}_{*}^{(n)}$ as well, write
\begin{align*}
e^{x_{*}} & =1+x_{*}+x'_{*}\\
e^{y_{*}} & =1+y_{*}+y'_{*}
\end{align*}
where $x'_{*}$ and $y'_{*}$ are built up from terms of degree $>1$.
Then, the ring homomrphism $\phi_{*}:\mathring{\Pi}_{*}\to\mathring{\Pi}_{*}$
defined by sending $x_{ij}\to x_{ij}+x_{ij}'$ and $y_{ij}\to y_{ij}+y_{ij}'$
gives rise to a ring homomorphism $M_{2}(\mathring{\Pi}_{*})\to M_{2}(\mathring{\Pi}_{*})$
that gives rise to a group homomorphism $\Psi_{*}:\hat{G}_{*}\to1+M_{2}(\mathring{\Pi}_{*},\mathring{Q}_{*1})$
defined by
\begin{align*}
 & 1+x_{*}\mapsto e^{x_{*}}\\
 & 1+y_{*}\mapsto e^{y_{*}}.
\end{align*}
By the Baker\textendash Campbell\textendash Hausdorff formula, one
has
\[
\min(\Psi_{*}(g))\in\mathring{L}_{*}^{(n)}
\]
for some $n$. But as $\Psi_{*}$ is originally induced by the ring
homomorphism $\phi_{*}$, one obtains that
\[
\min(g)=\min(\Psi_{*}(g))\in\mathring{L}_{*}^{(n)}
\]
as claimed in Proposition \ref{prop:min-1}.

\subsection{\label{subsec:Zubkov's proof}An algorithm for constructing a pro-$p$
identity}

Fix $d=2$, so that\foreignlanguage{american}{
\[
x_{*}=\left(\begin{array}{cc}
x_{11} & x_{12}\\
x_{21} & x_{22}
\end{array}\right),\,\,\,\,\,\,y_{*}=\left(\begin{array}{cc}
y_{11} & y_{12}\\
y_{21} & y_{22}
\end{array}\right)\in M_{2}(\Pi_{*}).
\]
}

In his paper, Zubkov gave the following formulas regarding the $2$-dimensional
generic matrices:\foreignlanguage{american}{
\begin{align*}
[x_{*},y_{*},x_{*},x_{*}] & =\alpha_{*}[x_{*},y_{*}]\,\,\,\,\textrm{for}\,\,\,\,\alpha_{*}=t(x_{*})^{2}-4\cdot\det(x_{*})\\{}
[x_{*},y_{*},x_{*},y_{*}]=[x_{*},y_{*},y_{*},x_{*}] & =\beta_{*}[x_{*},y_{*}]\,\,\,\,\textrm{for}\,\,\,\,\beta_{*}=2t(x_{*}y_{*})-t(x_{*})t(y_{*})\\{}
[x_{*},y_{*},y_{*},y_{*}] & =\gamma_{*}[x_{*},y_{*}]\,\,\,\,\textrm{for}\,\,\,\,\gamma_{*}=t(y_{*})^{2}-4\cdot\det(y_{*})
\end{align*}
where $t(x_{*}),t(y_{*}),t(x_{*}y_{*})$ are the traces of $x_{*},y_{*},x_{*}y_{*}$
respectively. It follows that 
\[
\mathring{L}_{*}^{(n)}=\begin{cases}
\sum_{r+s+t=(n-2)/2}(\mathbb{Q}_{p}\alpha_{*}^{r}\beta_{*}^{s}\gamma_{*}^{t}[x_{*},y_{*}]) & n=\textrm{even}\\
\sum_{r+s+t=(n-3)/2}(\mathbb{Q}_{p}\alpha_{*}^{r}\beta_{*}^{s}\gamma_{*}^{t}[x_{*},y_{*},x_{*}]+\mathbb{Q}_{p}\alpha_{*}^{r}\beta_{*}^{s}\gamma_{*}^{t}[x_{*},y_{*},y_{*}]) & n=\textrm{odd}
\end{cases}
\]
and we have a similar description for $L_{*}^{(n)}$. By showing that
the above sums are actually direct, Zubkov deduces 
\[
\textrm{rank}_{\mathbb{Z}_{p}}(L_{*}^{(n)})=\dim_{\mathbb{Q}_{p}}(\mathring{L}_{*}^{(n)})=\begin{cases}
n(n+2)/8 & n\,\,\,is\,\,\,even\\
(n-1)(n+1)/4 & n\,\,\,is\,\,\,odd.
\end{cases}
\]
}

Now, l\foreignlanguage{american}{et $\hat{F}=\hat{F}_{1},\hat{F}_{2},\hat{F}_{3},...$
be the lower central series of $\hat{F}$, and let $\Upsilon_{n}(\hat{F})=\hat{F}_{n}/\hat{F}_{n+1}$.
Then, the abelain groups $\Upsilon_{n}(\hat{F})\cong\mathbb{Z}_{p}^{l_{2}(n)}$
are free $\mathbb{Z}_{p}$-modules of the rank given by the Witt formula
(See \cite{key-12}, Proposition 2.7)
\begin{equation}
l_{2}(n)=\frac{1}{n}\sum_{m|n}\mu(m)\cdot2^{\frac{n}{m}}\label{eq:l2n}
\end{equation}
where $\mu$ is the Mobius function. }

Denote
\begin{align*}
\omega_{n}(\hat{G}_{*}) & =\ker(\hat{G}_{*}\to GL_{2}(\Pi_{*}/Q_{*n}))\\
\Omega_{n}(\hat{G}_{*}) & =\omega_{n}(\hat{G})/\omega_{n+1}(\hat{G}_{*}).
\end{align*}
\foreignlanguage{american}{It is easy to verify that in general, $\pi_{*}(\hat{F}_{n})\subseteq\omega_{n}(\hat{G}_{*})$.
Hence, for every $n$ we have a natural map
\[
\Upsilon_{n}(\hat{F})\to\Omega_{n}(\hat{G}_{*}).
\]
Now, notice that in general, by Proposition \ref{prop:min-1}, one
can view $\Omega_{n}(\hat{G}_{*})$ as an abelian subgroup of $\mathring{L}_{*}^{(n)}\cap A$
such that
\begin{equation}
\Upsilon_{n}(\hat{F})\twoheadrightarrow L_{*}^{(n)}\leq\Omega_{n}(\hat{G}_{*})\leq\mathring{L}_{*}^{(n)}\cap A_{*}.\label{eq:inequality-1}
\end{equation}
}

\selectlanguage{american}%
Zubkov proves the following proposition:
\selectlanguage{english}%
\begin{prop}
\label{prop:equality}When $p\neq2$, we have $L_{*}^{(n)}=\mathring{L}_{*}^{(n)}\cap A_{*}$. 
\end{prop}

We remark that this is not true when $p=2$, and we will elaborate
about it more in $\mathsection$\ref{sec:char=00003D0}. As a corollary
of Proposition \ref{prop:equality} and equation (\ref{eq:inequality-1})
we have:
\begin{cor}
\label{cor:surjection}When $p\neq2$, we have
\[
\Upsilon_{n}(\hat{F})\cong\mathbb{Z}_{p}^{l_{2}(n)}\twoheadrightarrow\Omega_{n}(\hat{G})\cong L_{*}^{(n)}\cong\mathbb{Z}_{p}^{m(n)}
\]
where $m(n)=n(n+2)/8$ if $n$ is even, and $m(n)=(n-1)(n+1)/4$ if
$n$ is odd.
\end{cor}

\selectlanguage{american}%
Now, let $p\neq2$. For $n=6$ we have $l_{2}(6)=9$ and $m(6)=6$\foreignlanguage{english}{.
Hence, Corollary \ref{cor:surjection} implies that there exists an
element $g\in\hat{F}_{6}-\hat{F}_{7}$ such that $\deg(\min(\pi_{*}(g)))\geq7$.
By Corollary \ref{cor:surjection}, there exists an element $h_{7}\in\hat{F}_{7}$
such that 
\[
\min(\pi_{*}(g))=-\min(\pi_{*}(h_{7})).
\]
Hence, $gh_{7}\in\hat{F}_{6}-\hat{F}_{7}$ and $\deg(\min(\pi_{*}(gh_{7})))\geq8$.
Again, by Corollary \ref{cor:surjection}, there exists an element
$h_{8}\in\hat{F}_{8}$ such that 
\[
\min(\pi_{*}(gh_{7}))=-\min(\pi_{*}(h_{8})).
\]
Hence, $gh_{7}h_{8}\in\hat{F}_{6}-\hat{F}_{7}$ and $\deg(\min(\pi_{*}(gh_{7}h_{8})))\geq9$.
Continuing with this algorithm one obtains a sequence of elements
$h_{i}\in\hat{F}_{i}$ such that 
\[
w\in\lim_{n\to\infty}gh_{7}h_{8}\cdot...\cdot h_{n}\in\hat{F}_{6}-\hat{F}_{7}
\]
coverges, but $\pi_{*}(w)=1\in\hat{G}_{*}$. In other words, $1\neq w\in\hat{F}$
serves as a pro-$p$ identity of $2\times2$ linear groups over pro-$p$
rings by Theorem \ref{thm:universal-1}.}
\selectlanguage{english}%

\section{\label{sec:char=00003D2}The case $\textrm{char}(\Delta)=2$}

\subsection{\label{subsec:universal 2}The Universal Representation}

\selectlanguage{american}%
In this section we use similar notation as in Section \ref{sec:Zubkov}
for some objects that are slightly different, but play a similar role
in this section. 

Let $x_{i,j}$ and $y_{i,j}$ for $1\leq i,j\leq2$ be free commuting
variables, and let
\[
\Lambda_{*}=(\mathbb{Z}/2\mathbb{Z})\left\langle \left\langle x_{i,j},y_{i,j}\,|\,1\leq i,j\leq2\right\rangle \right\rangle 
\]
be the associative ring (with identity) of formal power series in
$x_{i,j}$ and $y_{i,j}$ over $\mathbb{Z}/2\mathbb{Z}$. Every element
in $\Lambda_{*}$ can be written as $f=\sum_{i=0}^{\infty}f_{i}$
where $f_{i}$ is homogeneous of degree $i$. The finite index ideals
\begin{equation}
\Lambda_{*}\vartriangleright P_{*n}=\left\{ f=\sum_{i=0}^{\infty}f_{i}\in\Lambda_{*}\,|\,f_{0},...,f_{n-1}=0\,,\deg(f_{i})=i\right\} \label{eq:ideals}
\end{equation}
serve as a basis of neighborhoods of zero for the profinite topology
of $\Lambda_{*}$, making $\Lambda_{*}$ a pro-$2$ ring, with $P_{*1}$
as its maximal ideal. 

Endowed with the topology that comes from the congruence ideals 
\[
M_{2}(\Lambda_{*},P_{*n})=\ker(M_{2}(\Lambda_{*})\to M_{2}(\Lambda_{*}/P_{*n}))
\]
as a basis of neighborhoods of zero, $M_{2}(\Lambda_{*})$ is a profinite
ring. It is easy to check that this topology makes the group $1+M_{2}(\Lambda_{*},P_{*1})$
a pro-$2$ group. 

Denoting the \uline{generic matrices }
\[
x_{*}=\left(\begin{array}{cc}
x_{11} & x_{12}\\
x_{21} & x_{22}
\end{array}\right),\,\,\,\,y_{*}=\left(\begin{array}{cc}
y_{11} & y_{12}\\
y_{21} & y_{22}
\end{array}\right)\in M_{2}(\Lambda_{*},P_{*1}).
\]
we get a natural (continuous) map $\pi_{*}:\hat{F}\to1+M_{2}(\Lambda_{*},P_{*1})$
defined by
\[
X\mapsto1+x_{*},\,\,\,\,\,\,Y\mapsto1+y_{*}
\]
where as before $\hat{F}=\left\langle \left\langle X,Y\right\rangle \right\rangle $
is the free pro-$2$ group generated by $X,Y$. We denote by $G_{*}=\left\langle 1+x_{*},1+y_{*}\right\rangle \subseteq1+M_{2}(\Lambda_{*},P_{*1})$
the (discrete) subgroup generated by $1+x_{*}$ and $1+y_{*}$, and
$\hat{G}_{*}\subseteq1+M_{2}(\Lambda_{*},P_{*1})$ its closure in
$1+M_{2}(\Lambda_{*},P_{*1})$. Adopting Zubkov's terminology, $\pi_{*}:\hat{F}\twoheadrightarrow\hat{G}_{*}$
is called \uline{the universal representation}.
\selectlanguage{english}%
\begin{prop}
\label{prop:universal}Let $\Delta$ be any pro-$2$ ring with $char(\Delta)=2$.
Then, every $1\neq w(X,Y)\in\ker\pi_{*}$ is a pro-$2$ identity of
$GL_{2}^{1}(\Delta)$. 
\end{prop}

\begin{rem}
For the purpose of the present paper, we stated Proposition \ref{prop:universal}
for the special case where $d=p=2$. However, as will be seen from
the proof below, under appropriate generalization of the universal
representation, one can state a similar proposition for any dimension
$d$ and prime $p$. 
\end{rem}

\begin{rem}
Notice that if we omit the assumption $char(\Delta)=2$, the proposition
is not necessarily true. Only elements that lie in the kernel of the
universal representation over $\mathbb{Z}_{2}$ (as defined by Zubkov),
and not over $\mathbb{Z}/2\mathbb{Z}$ (as we defined), will be identities
of $GL_{2}^{1}(\Delta)$ for any pro-$2$ ring $\Delta$.
\end{rem}

\begin{proof}
(of Proposition \ref{prop:universal}) Let $\sigma:\hat{F}\to GL_{2}^{1}(\Delta)$
be a (continuous) map. By assumption, $\Delta=\underleftarrow{\lim}\Delta_{n}$,
where $\Delta_{n}$ are commutative finite local rings with $char(\Delta_{n})=2$.
Hence, without loss of generality, we can assume that $\Delta$ is
finite. Let $I$ be the maximal ideal of $\Delta,$ and let $n$ be
a natural number satisfying $I^{n}=0$. By definition, we can write
$\sigma(X)=1+\alpha$ and $\sigma(Y)=1+\beta$ for some $\alpha,\beta\in\ker(M{}_{2}(\Delta)\to M_{2}(\Delta/I))$.
Hence, as $x_{*},y_{*}$ are generic matrices over $\mathbb{Z}/2\mathbb{Z}$,
there is a homomorphism (the map in the right is not necessarily surjective)
\[
\tau:\Lambda_{*}\twoheadrightarrow\Lambda_{*}/P_{*n}\to\Delta
\]
defined by sending the entries of $x_{*}$ and $y_{*}$ to the corresponding
entries of $\alpha$ and $\beta$. Hence, $\tau$ induces a continuous
homorphism $\tau_{*}:\hat{G}_{*}\to GL_{2}^{1}(\Delta)$ such that
$\tau_{*}\circ\pi_{*}=\sigma$, implying that $\ker\sigma\subseteq\ker\pi_{*}$,
as required.
\end{proof}
Proposition \ref{prop:universal} shows that \foreignlanguage{american}{Theorem
\ref{thm:main-main-1}} boils down to the following theorem:
\selectlanguage{american}%
\begin{thm}
\label{thm:main}The universal representation $\pi_{*}:\hat{F}\twoheadrightarrow\hat{G}_{*}$
is not injective.
\end{thm}

\subsection{\label{subsec:reduction}Reduction of Theorem \ref{thm:main}}

\selectlanguage{english}%
We want to replace the generic matrices $x_{*},y_{*}$ by matrices
$x,y$ that satisfy the condition $\det(x)=\det(y)=0$. 

Let $h$ be a rational function in $x_{i,j},y_{i,j}$ over $\mathbb{Z}/2\mathbb{Z}$,
and denote the discrete ring
\[
\Lambda_{\#}=(\mathbb{Z}/2\mathbb{Z})\left\langle x_{i,j},y_{i,j}\,|\,1\leq i,j\leq2\right\rangle \leq\Lambda_{*}.
\]
We say that $h$ is \uline{homogeneous} if there exist homogeneous
polynomials $f,g\in\Lambda_{\#}$, $g\neq0$, such that $h=\frac{f}{g}$.
In this case we define 
\[
\deg\left(h\right)=\deg(f)-\deg(g).
\]
Clearly, $\deg\left(h\right)$ is well defined, i.e. if $\frac{f_{1}}{g_{1}}=\frac{f_{2}}{g_{2}}$
then $\deg\left(\frac{f_{1}}{g_{1}}\right)=\deg\left(\frac{f_{2}}{g_{2}}\right)$.
Now, consider the set of all power series of the from
\[
\mathbb{F}=\left\{ \sum_{i=m}^{\infty}h_{i}\,|\,\,h_{i}\,\,\textrm{is\,\,rational\,\,homogeneous\,\,of}\,\,\deg\left(h_{i}\right)=i,\,\,m\in\mathbb{Z}\right\} .
\]
It is easy to see that $\mathbb{F}$ has a natural ring structure.
Moreover, let $0\neq h\in\mathbb{F}$, and write $h=\sum_{i=m}^{\infty}h_{i}$
where $h_{m}\neq0$. Then
\[
h=h_{m}(1+a)\,\,\,\,\textrm{where}\,\,\,\,a=h_{m}^{-1}\sum_{i=m+1}^{\infty}h_{i}.
\]
As $h_{m}$ and $1+a$ are invertible in $\mathbb{F}$, we obtain
that $h$ is invertible as well. It follows that $\mathbb{F}$ is
actually a field that contains $\Lambda_{*}$. \foreignlanguage{american}{Consider
now the following quadratic polynomials in the variables $\mu$ and
$\nu$ over the field $\mathbb{F}$:
\begin{align*}
p(\mu) & =\det((1+x_{*})(1+\mu\cdot1)-1)\\
 & =(1+t(x_{*})+\det(x_{*}))\cdot\mu^{2}+t(x_{*})\cdot\mu+\det(x_{*})\\
\\
q(\nu) & =\det((1+y_{*})(1+\nu\cdot1)-1)\\
 & =(1+t(y_{*})+\det(y_{*}))\cdot\nu^{2}+t(y_{*})\cdot\nu+\det(y_{*})
\end{align*}
where $t(x_{*})$ and $t(y_{*})$ are the traces of $x_{*}$ and $y_{*}$
respectively. }
\begin{lem}
\label{lem:pro-2 ring}Let $\mu_{1},\mu_{2}$ be quadratic elements
over $\mathbb{F}$ such that 
\begin{equation}
\mu_{i}^{2}\in\mu_{i}\cdot(P_{*1}-P_{*2})+P_{*2}\,\,,\,\,\,\,i=1,2.\label{eq:condition}
\end{equation}
where $P_{*1},P_{*2}\vartriangleleft\Lambda_{*}$ are the ideals defined
in (\ref{eq:ideals}). Then, the set
\begin{equation}
\Lambda=\Lambda_{*}+\mu_{1}\cdot\Lambda_{*}+\mu_{2}\cdot\Lambda_{*}+\mu_{1}\mu_{2}\cdot\Lambda_{*}\label{eq:ring}
\end{equation}
is a pro-$2$ ring with a maximal ideal $P=P_{*1}+\mu_{1}\cdot\Lambda_{*}+\mu_{2}\cdot\Lambda_{*}+\mu_{1}\mu_{2}\cdot\Lambda_{*}$.
Moreover, for every $n\geq0$ we have $P^{n}\cap\Lambda_{*}=P_{*n}$.
In particular, the topology that $\Lambda$ induces on $\Lambda_{*}$
coincides with the topology of $\Lambda_{*}$ defined by $P_{*n}$. 
\end{lem}

\begin{proof}
By (\ref{eq:condition}) it is clear that $\Lambda$ is indeed a ring,
and that $P\vartriangleleft\Lambda$ is an ideal. For an element $h=\sum_{i=m}^{\infty}h_{i}\in\mathbb{F}$,
$m\in\mathbb{Z}$, such that $h_{i}$ are homogeneous of degree $i$
and $h_{m}\neq0$, we denote $\min(h)=h_{m}$. We have three cases
to consider:\\

\uline{Case 1:}\textbf{ }$\mu_{1}\notin\mathbb{F}$ and $\mu_{2}\notin\mathbb{F}(\mu_{1})$.
In this case, the sum in (\ref{eq:ring}) is direct, and all the properties
in the lemma follow easily.\\

\uline{Case 2:} $\mu_{1},\mu_{2}\in\mathbb{F}$. From (\ref{eq:condition})
we get that for each $i=1,2$ we either have $\mu_{i}=0$ or $\deg(\min(\mu_{i}))\geq1$.
Hence, the profinite topology on $\Lambda$ is induced by the degree
on $\mathbb{F}$ and the lemma follows easily.\\

\uline{Case 3:} $\mu_{1}\notin\mathbb{F}$ and $\mu_{2}\in\mathbb{F}(\mu_{1})$.
In this case we can write
\[
\mu_{2}=\alpha+\beta\cdot\mu_{1}\,\,,\,\,\,\,\mu_{1}^{2}=\gamma\mu_{1}+\delta
\]
for some $\alpha,\beta\in\mathbb{F}$, $\gamma\in P_{*1}-P_{*2}$
and $\delta\in P_{*2}$. Hence, (\ref{eq:condition}) implies
\[
\alpha^{2}+\beta^{2}\cdot(\gamma\mu_{1}+\delta)=\alpha^{2}+\beta^{2}\cdot\mu_{1}^{2}=\mu_{2}^{2}\in(\alpha+\beta\cdot\mu_{1})\cdot P_{*1}+P_{*2}.
\]
Hence ({*}) $\alpha^{2}\in\beta^{2}\delta+\alpha P_{*1}+P_{*2}$ and
$\beta^{2}\gamma\in\beta\cdot P_{*1}$. As $\gamma\in P_{*1}-P_{*2}$,
the latter implies that either $\beta=0$ or $\deg(\min(\beta))\geq0$.
If $\beta\neq0$ with $\deg(\min(\beta))\geq0$ we get that ({*})
implies $\deg(\min(\alpha))\geq1$ or $\alpha=0$. Writing
\begin{align*}
\Lambda & =\Lambda_{*}+\mu_{1}\cdot\Lambda_{*}+(\alpha+\beta\cdot\mu_{1})\cdot\Lambda_{*}+\mu_{1}(\alpha+\beta\cdot\mu_{1})\cdot\Lambda_{*}\\
 & =(\Lambda_{*}+\alpha\cdot\Lambda_{*}+\beta\delta\cdot\Lambda_{*})\oplus\mu_{1}(\Lambda_{*}+\beta\cdot\Lambda_{*}+\alpha\cdot\Lambda_{*}+\beta\gamma\cdot\Lambda_{*})
\end{align*}
the lemma follows again easily. The case $\beta=0$ is similar.
\end{proof}
Now, let $\overline{\mu}$ and $\overline{\nu}$ be roots of $p(\mu)$
and $q(\nu)$ respectively. Notice that the elements $1+t(x_{*})+\det(x_{*})$
and $1+t(y_{*})+\det(y_{*})$ are invertible over $\Lambda_{*}$,
and we have
\begin{align*}
\overline{\mu}^{2} & =(1+t(x_{*})+\det(x_{*}))^{-1}\cdot(\overline{\mu}\cdot t(x_{*})+\det(x_{*}))\in\overline{\mu}\cdot(P_{*1}-P_{*2})+P_{*2}\\
\overline{\nu}^{2} & =(1+t(y_{*})+\det(y_{*}))^{-1}\cdot(\overline{\nu}\cdot t(y_{*})+\det(y_{*}))\in\overline{\nu}\cdot(P_{*1}-P_{*2})+P_{*2}.
\end{align*}
Hence, by Lemma \ref{lem:pro-2 ring}, the ring $\Lambda=\Lambda_{*}+\overline{\mu}\cdot\Lambda_{*}+\overline{\nu}\cdot\Lambda_{*}+\overline{\mu}\cdot\overline{\nu}\cdot\Lambda_{*}$
is a pro-$2$ ring with a maximal ideal $P=P_{*1}+\overline{\mu}\cdot\Lambda_{*}+\overline{\nu}\cdot\Lambda_{*}+\overline{\mu}\cdot\overline{\nu}\cdot\Lambda_{*}$.
In addition, $P^{n}\cap\Lambda_{*}=P_{*n}$ for any $n\geq1$ so\foreignlanguage{american}{
the topology of $\Lambda_{*}\subseteq\Lambda$ that is induced by
the topology of $\Lambda$ coincides with the topology of $\Lambda_{*}$
defined by the basis $P_{*n}$. }
\begin{rem}
An explicit computation shows that $\overline{\mu}$ and $\overline{\nu}$
are not in $\Gamma$ and are independent, so our specific case corresponds
to the first case in Lemma \ref{lem:pro-2 ring}. However, Lemma \ref{lem:pro-2 ring}
allows us to avoid this computation.
\end{rem}

\selectlanguage{american}%
Endowed with the topology that comes from the congruence ideals 
\[
M_{2}(\Lambda,P^{n})=\ker(M_{2}(\Lambda)\to M_{2}(\Lambda/P^{n}))
\]
as a basis of neighborhoods of zero, $M_{2}(\Lambda)$ is also a profinite
ring, that contains the profinite ring $M_{2}(\Lambda_{*})$. Also
here, this topology makes the group $1+M_{2}(\Lambda,P)$ a pro-$2$
group that contains the pro-$2$ subgroup $1+M_{2}(\Lambda_{*},P_{*1})$.
Notice that as $P^{n}\cap\Lambda_{*}=P_{*n}$, the profinite topology
that $1+M_{2}(\Lambda,P)$ induces on $1+M_{2}(\Lambda_{*},P_{*1})$
coincides with the profinite topology of $1+M_{2}(\Lambda_{*},P_{*1})$
defined by the basis $1+M_{2}(\Lambda_{*},P_{*n})$.

We are now ready to define the \uline{pseudo-generic} matrices
\begin{align*}
x & =(1+x_{*})(1+\overline{\mu}\cdot1)-1\in M_{2}(\Lambda,P)\\
y & =(1+y_{*})(1+\overline{\nu}\cdot1)-1\in M_{2}(\Lambda,P).
\end{align*}
By the construction of $x$ and $y$ we have $\det(x)=\det(y)=0$.

\selectlanguage{english}%
Let now\foreignlanguage{american}{ $G=\left\langle 1+x,1+y\right\rangle \subseteq1+M_{2}(\Lambda,P)$
be the (discrete) group generated by $1+x$ and $1+y$, and let $\hat{G}\subseteq1+M_{2}(\Lambda,P)$
be its closure. Notice that by the discussion above, both $\hat{G}_{*}$
and $\hat{G}$ are embedded in $1+M_{2}(\Lambda,P)$, and their topology
is induced by the one of $1+M_{2}(\Lambda,P)$.}

\selectlanguage{american}%
Now, by the definition of the pseudo-generic matrices
\begin{align*}
1+x & =(1+x_{*})(1+\overline{\mu}\cdot1)\\
1+y & =(1+y_{*})(1+\overline{\nu}\cdot1)
\end{align*}
where $1+\overline{\mu}\cdot1$ and $1+\overline{\nu}\cdot1$ are
central in $M_{2}(\Lambda)$. Hence, for every commutator element
in the discrete free group $w(X,Y)\in F'$ one has
\[
w(1+x,1+y)=w(1+x_{*},1+y_{*})\in G'_{*}\subseteq1+M_{2}(\Lambda_{*},P_{*1})\subseteq1+M_{2}(\Lambda,P).
\]
As we saw that both topologies of $\hat{G}_{*}$ and $\hat{G}$ are
induced by the topology of $1+M_{2}(\Lambda,P)$, it follows that
actually
\[
\hat{G}'=\hat{G}'_{*}\subseteq1+M_{2}(\Lambda_{*},P_{*1})\subseteq1+M_{2}(\Lambda,P)
\]
for the commutator subgroups of $\hat{G}_{*}$ and $\hat{G}$. Similarly,
the lower central series of $\hat{G}$ and $\hat{G}{}_{*}$ are equal
(apart from the first term). Hence, in order to prove Theorem \ref{thm:main},
it is enough to prove that:
\begin{thm}
\label{thm:main-1}Let \foreignlanguage{english}{$\pi:\hat{F}\to\hat{G}$
defined by $X\mapsto1+x$, $Y\mapsto1+y$. Then, the restriction}
$\pi|_{\hat{F}'}:\hat{F}'\to\hat{G}'$ is not injective.
\end{thm}

From now on, we assume that Theorem \ref{thm:main-1} is false, and
that $\pi$ is injective. In particular, we assume that $\hat{G}'$
is isomorphic to $\hat{F}'$ through $\pi$.

\subsection{\label{subsec:Some-useful-lemmas}Some useful lemmas}

\selectlanguage{english}%
We state some general properties of $2\times2$ matrices. By the Cayley\textendash Hamilton
theorem, for every $a\in M_{2}(\Lambda)$, we have $a^{2}+t(a)a+\det(a)\cdot1=0$
where $t(a)$ is the trace of $a$. As a corollary of that, we have
the following lemmas:
\begin{lem}
\label{lem:a,b}Let $a,b\in M_{2}(\Lambda)$. Then:
\begin{enumerate}
\item $[a,b]=t(a)b+t(b)a+\left(t(ab)+t(a)t(b)\right)\cdot1.$
\item $[a,b,a]=t(a)[a,b]$. 
\item If $t(a)=0$, then $a^{2}=\det(a)\cdot1\in M_{2}(\Lambda)$ is central.
\item The trace $t([a,b]b^{n})=0$ for every $n\geq0$.
\end{enumerate}
\end{lem}

\begin{proof}
Applying the Cayley\textendash Hamilton theorem for $a+b$ one has
\begin{align*}
0 & =(a+b)^{2}+t(a+b)(a+b)+\det(a+b)\cdot1\\
 & =a^{2}+t(a)a+b^{2}+t(b)b+ab+ba+t(a)b+t(b)a+\det(a+b)\cdot1.
\end{align*}
Subtracting the equations $a^{2}+t(a)a+\det(a)\cdot1=0$ and $b^{2}+t(b)b+\det(b)\cdot1=0$
it follows that
\begin{align*}
[a,b]=ab+ba & =t(a)b+t(b)a+\left(\det(a)+\det(b)+\det(a+b)\right)\cdot1\\
 & =t(a)b+t(b)a+\left(t(ab)+t(a)t(b)\right)\cdot1
\end{align*}
so we get Part 1. Part 2 is an immediate consequence of Part 1. Part
3 is an immediate consequence of the Cayley\textendash Hamilton theorem.

For Part 4 we use induction on $n$. For $n=0$ the claim is easy,
and for $n=1$ it follows from the observation $[a,b]b=[ab,b]$. For
$n\geq2$: by the Cayley\textendash Hamilton theorem we have $b^{2}=t(b)b+\det(b)\cdot1$.
Hence, by the induction hypothesis
\[
t([a,b]b^{n})=t(b)\cdot t([a,b]b^{n-1})+\det(b)\cdot t([a,b]b^{n-2})=0
\]
as required. 
\end{proof}
\begin{lem}
\label{lem:x,y}For the \foreignlanguage{american}{pseudo-generic
matrices} $x,\,y$ we have:
\begin{enumerate}
\item $x^{2}=t(x)x,\,\,\,\,y^{2}=t(y)y,\,\,\,\,(xy)^{2}=t(xy)xy,\,\,\,\,(yx)^{2}=t(xy)yx$.
\item $xyx=t(xy)x,\,\,\,\,yxy=t(xy)y$.
\item $[x,y]^{2}=(t(xy)^{2}+t(x)t(y)t(xy))\cdot1$.
\end{enumerate}
\end{lem}

\begin{proof}
Part 1 follows from the Cayley\textendash Hamilton theorem and the
property $\det(x)=\det(y)=0$. The identity $xyx=t(xy)x$ follows
from the previous part and Part 1 of Lemma \ref{lem:a,b} by the following
computation

\begin{align*}
xyx & =(yx+t(x)y+t(y)x+\left(t(xy)+t(x)t(y)\right)\cdot1)x\\
 & =yx^{2}+t(x)yx+t(y)x^{2}+t(xy)x+t(x)t(y)x=t(xy)x.
\end{align*}
The identity $yxy=t(xy)y$ follows similarly. Part 3 follows from
the previous properties by the following computation
\begin{align*}
[x,y]^{2}= & (xy+yx)(xy+yx)\\
= & t(xy)(xy+yx)+t(xy)t(x)y+t(xy)t(y)x=t(xy)(t(xy)+t(x)t(y))\cdot1.
\end{align*}
\end{proof}
\begin{rem}
As the expression $[x,y]^{2}$ is central in $M_{2}(\Lambda)$, sometimes
we will consider it as an element of $\Lambda$ and just write $[x,y]^{2}=(t(xy)^{2}+t(x)t(y)t(xy))$.
\end{rem}

\selectlanguage{american}%

\subsection{\label{subsec:ring of}The ring of the Pseudo-Generic Matrices}
\selectlanguage{english}%
\begin{defn}
We define the (discrete) subrings (with identity) of $\Lambda$ and
$M_{2}\left(\Lambda\right)$ 
\begin{align*}
S & =\left\langle t(x),\,t(y),\,[x,y]^{2}\right\rangle \subseteq\Lambda\\
T & =\left\langle t(x),\,t(y),\,t(xy)\right\rangle \subseteq\Lambda\\
R & =\left\langle x,\,y,\,T\cdot1\right\rangle \subseteq M_{2}\left(\Lambda\right).
\end{align*}
The ring $R$ will be called \uline{the ring of the pseudo generic
matrices}.
\end{defn}

The main purpose of this subsection is to investigate the ring $R$
and study some of its propeties which will be needed throughout the
proof of \foreignlanguage{american}{Theorem \ref{thm:main-1}.} 
\begin{prop}
\label{prop:free}The ring $T$ is freely generated by $t(x),\,t(y),\,t(xy)$
as a commutative algebra over $\mathbb{Z}/2\mathbb{Z}$. 
\end{prop}

\begin{proof}
Consider the free commutative variables $\lambda,\,\theta,\,\vartheta$
over $\mathbb{Z}/2\mathbb{Z}$, and define 
\[
\overline{x}=\left(\begin{array}{cc}
\lambda & \vartheta-\lambda\theta\\
0 & 0
\end{array}\right),\,\,\,\,\overline{y}=\left(\begin{array}{cc}
\theta & 0\\
1 & 0
\end{array}\right).
\]
Denote $\overline{T}=\left\langle t(\overline{x}),\,t(\overline{y}),\,t(\overline{x}\cdot\overline{y})\right\rangle $.
As $\det(\overline{x})=\det(\overline{y})=0$, it is easy to verify
that we have a natural homomorphism of discrete rings 
\[
\Lambda\geq(\mathbb{Z}/2\mathbb{Z})\left\langle x_{ij},y_{ij},\overline{\mu},\overline{\nu}\,|\,1\leq i,j\leq2\right\rangle \to(\mathbb{Z}/2\mathbb{Z})\left\langle \lambda,\theta,\vartheta\right\rangle 
\]
 defined by
\[
\begin{cases}
x_{11}\mapsto\lambda & y_{11}\mapsto\theta\\
x_{12}\mapsto\vartheta-\lambda\theta\,\,\,\,\,\,\,\, & y_{12}\mapsto0\\
x_{21}\mapsto0 & y_{21}\mapsto1\\
x_{22}\mapsto0 & y_{22}\mapsto0\\
\overline{\mu}\mapsto0 & \overline{\nu}\mapsto0
\end{cases}
\]
that induces a natural ring homomorphism from the discrete ring generated
by the pseudo-generic matrices $x,\,y$ to the discrete ring generated
by $\overline{x},\,\overline{y}$ sending $x\mapsto\overline{x}$
and $y\mapsto\overline{y}$. Hence, we have a natural ring homomorphism
$T\to\overline{T}$ by 
\begin{align*}
t(x) & \mapsto t(\overline{x})=\lambda\\
t(y) & \mapsto t(\overline{y})=\theta\\
t(xy) & \mapsto t(\overline{x}\cdot\overline{y})=\vartheta.
\end{align*}
So as the ring generated by $\lambda,\,\theta,\,\vartheta$ is freely
generated by them as a commutative algebra over $\mathbb{Z}/2\mathbb{Z}$,
the same is valid for $T$. 
\end{proof}
As a corollary of Part 3 in Lemma \ref{lem:x,y}, and Proposition
\ref{prop:free} we have:
\begin{prop}
\label{prop:ST}The ring $S$ is contained in $T$ and it is freely
geberated by $t(x),\,t(y),\,[x,y]^{2}$ as a commutative algebra over
$\mathbb{Z}/2\mathbb{Z}$. In addition, $T$ is freely generated by
$1$ and $t(xy)$ as an $S$-module. 
\end{prop}

We also have:
\begin{prop}
\label{prop:finitely generated}The ring $R$ is freely generated
by $1,\,x,\,y,\,xy$ as a $T$-module.
\end{prop}

\begin{proof}
It follows from Lemma \ref{lem:x,y} that indeed $R$ is generated
by $1,\,x,\,y,\,xy$ as a $T$-module. We have to show that the way
to write $a=\alpha+\beta\cdot x+\gamma\cdot y+\delta\cdot xy$ for
$\alpha,\beta,\gamma,\delta\in T$ is unique. Observe that as $t([x,y])=t([x,y]x)=t([x,y]y)=0$
and $x[x,y]x=y[x,y]y=0$, given such $a$, we have
\begin{align*}
t(x[x,y]a[x,y]y) & =\alpha\cdot t(xy)[x,y]^{2}\\
t([x,y]ya) & =\beta\cdot[x,y]^{2}\\
t([x,y]ax) & =\gamma\cdot[x,y]^{2}\\
t([x,y]a) & =\delta\cdot[x,y]^{2}.
\end{align*}
Hence, as $T$ is a domain (as a free commutative ring by Proposition
\ref{prop:free}), given such $a$, we can uniquely restore its coefficients
in $T$, as required.
\end{proof}
Denote now the (discrete) ring 
\[
R_{*}=\left\langle x_{*},\,y_{*},\,t(x_{*})\cdot1,\,t(y_{*})\cdot1,\,t(x_{*}y_{*})\cdot1\right\rangle \subseteq M_{2}\left(\Lambda_{*}\right).
\]
Then $R_{*}$, inheriting the degree of $\Lambda_{*}$, can be seen
as $R_{*}=\oplus_{n=0}^{\infty}R_{*}^{(n)}$ where $R_{*}^{(n)}$
is the additive subgroup of $R_{*}$ of homogeneous elements of degree
$n$. This direct sum makes $R_{*}$ a graded ring. As $x_{*},y_{*}$
are generic matrices, we can define\foreignlanguage{american}{ a map
$\tau:R_{*}\to R$ such that $x_{*}\mapsto x$ and $y_{*}\mapsto y$.
Denote $R^{(n)}=\tau(R_{*}^{(n)})$. }
\begin{prop}
\label{cor:graded}One has $R=\oplus_{n=0}^{\infty}R^{(n)}$, which
makes $R$ a graded ring.
\end{prop}

\begin{proof}
Denote $T^{(n)}=R^{(n)}\cap T\cdot1$ for $n\geq0$, and $T^{(-1)}=T^{(-2)}=0$.
Using Lemma \ref{lem:x,y} it is easy to check that 
\[
R^{(n)}=\tau(R_{*}^{(n)})=T^{(n)}+T^{(n-1)}\cdot x+T^{(n-1)}\cdot y+T^{(n-2)}\cdot xy.
\]
Hence, by Proposition \ref{prop:finitely generated} we have
\begin{align*}
R^{(n)} & =T^{(n)}\oplus T^{(n-1)}\cdot x\oplus T^{(n-1)}\cdot y\oplus T^{(n-2)}\cdot xy.
\end{align*}
By Proposition \ref{prop:free} it is clear that $T\cdot1=\oplus_{n=0}^{\infty}T^{(n)}$.
Hence 
\begin{align*}
R & =T\oplus T\cdot x\oplus T\cdot y\oplus T\cdot xy\\
 & =\oplus_{n=0}^{\infty}T^{(n)}\oplus_{n=0}^{\infty}T^{(n)}\cdot x\oplus_{n=0}^{\infty}T^{(n)}\cdot y\oplus_{n=0}^{\infty}T^{(n)}\cdot xy=\oplus_{n=0}^{\infty}R^{(n)}
\end{align*}
 as required.
\end{proof}
The rest of this subsection deals with objects which are related to
the ring of the pseudo generic matrices. The motivation for dealing
with these objects will be clear later.
\begin{defn}
Let $J$ be the $T$-module generated by
\[
[x,y]x,\,[x,y]y,\,[x,y]^{2},\,[x,y]xy.
\]
For any $n\geq0$ we define the following $T$-submodules of $J$:
\begin{align*}
J_{n} & =t(x)^{n}J=t(x)^{n}(T[x,y]x+T[x,y]y+T[x,y]^{2}+T[x,y]xy)\subseteq J\\
\overline{J}_{n} & =t(x)^{n}(T[x,y]x+T[x,y]y)\subseteq J_{n}\subseteq J\\
\overline{C}_{n} & =t(x)^{n}T[x,y]^{2}\subseteq J_{n}\subseteq J.
\end{align*}
\end{defn}

\begin{prop}
\label{prop:J - f.g.}The $T$-module $J$ is a two sided ideal of
$R$. Moreover, $J$ is freely generated as a $T$-module by 
\[
[x,y]x,\,[x,y]y,\,[x,y]^{2},\,[x,y]xy.
\]
\end{prop}

\begin{proof}
For showing that $J$ is a two sided ideal of $R$, it is enough to
show that $J$ is closed under multiplication by $x,\,y$. Indeed,
for the generators of $J$ as a $T$-module we have
\begin{align*}
[x,y]x\cdot x & =t(x)[x,y]x\in J\\{}
[x,y]y\cdot x & =[x,y]^{2}+[x,y]xy\in J\\{}
[x,y]^{2}\cdot x & =t(xy)[x,y]x+t(x)[x,y]^{2}+t(x)[x,y]xy\in J\\{}
[x,y]xy\cdot x & =t(xy)[x,y]x\in J\\
\\
x\cdot[x,y]x & =0\in J\\
x\cdot[x,y]y & =(t(x)[x,y]+[x,y]x)\cdot y\in J\\
x\cdot[x,y]^{2} & =[x,y]^{2}\cdot x\in J\\
x\cdot[x,y]xy & =0\in J.
\end{align*}
The computation for $y$ is similar, so $T$ is a two-sided ideal
of $R$. 

Now, let $a\in J$. We want to show that the way to write 
\[
a=\alpha\cdot[x,y]x+\beta\cdot[x,y]y+\gamma\cdot[x,y]^{2}+\delta\cdot[x,y]xy
\]
for $\alpha,\beta,\gamma,\delta\in T$ is unique. Using the properties
in Lemma \ref{lem:x,y} one can verify that
\begin{align*}
t([x,y]ya) & =(t(xy)+t(x)t(y))[x,y]^{2}\alpha\\
t([x,y]xa) & =(t(xy)+t(x)t(y))[x,y]^{2}\beta\\
t(xa[x,y]y) & =[x,y]^{4}\gamma\\
t(a) & =[x,y]^{2}\delta
\end{align*}
Hence, as in the proof of Proposition \ref{prop:finitely generated},
it follows that $J$ is freely generated by $[x,y]x$, $[x,y]y$,
$[x,y]^{2}$ and $[x,y]xy$ over $T$, as required. 
\end{proof}
\begin{cor}
\label{cor:two sided}For every $n\geq0$, the $T$-module $J_{n}$
is a two sided ideal of $R$, and 
\[
J_{n}=\overline{J}_{n}\oplus\overline{C}_{n}\oplus t(x)^{n}T\cdot[x,y]xy.
\]
\end{cor}

\begin{cor}
\label{cor:trace=00003D0}Let $a\in J_{n}$. If $t(a)=0$, then $a\in\overline{J}_{n}+\overline{C}_{n}$.
I.e.
\[
a=\alpha\cdot[x,y]x+\beta\cdot[x,y]y+\gamma\cdot[x,y]^{2}
\]
for some $\alpha,\beta,\gamma\in t(x)^{n}T$. 
\end{cor}

As $T$ is a free commutative algebra over $\mathbb{Z}/2\mathbb{Z}$
(Proposition \ref{prop:free}), $T$ is a unique factorization domain.
Hence, by Proposition \ref{prop:J - f.g.} and Part 3 of Lemma \ref{lem:x,y},
we have:
\begin{cor}
\label{prop:iff}The $T$-modules $J_{n},\,\overline{J}_{n},\,\overline{C}_{n}$
are free, and for every $a\in J$ we have
\begin{align*}
a\in J_{n} & \Leftrightarrow t(x)a\in J_{n+1}\Leftrightarrow t(y)a,\,[x,y]^{2}a\in J_{n}\\
a\in\overline{J}_{n} & \Leftrightarrow t(x)a\in\overline{J}_{n+1}\Leftrightarrow t(y)a,\,[x,y]^{2}a\in\overline{J}_{n}\\
a\in\overline{C}_{n} & \Leftrightarrow t(x)a\in\overline{C}_{n+1}\Leftrightarrow t(y)a,\,[x,y]^{2}a\in\overline{C}_{n}.
\end{align*}
\end{cor}

\begin{defn}
We define $\mathring{S}=\left\langle t(y),\,[x,y]^{2}\right\rangle $
to be the subring of $S$ generated by $t(y)$ and $[x,y]^{2}$, and
define $\mathring{T}=\mathring{S}+t(xy)\mathring{S}$.
\end{defn}

From Proposition \ref{prop:ST} one deduces that $S=\oplus_{i=0}^{\infty}t(x)^{i}\cdot\mathring{S}$,
and therefore 
\[
T=S\oplus t(xy)S=\oplus_{i=0}^{\infty}t(x)^{i}(\mathring{S}\oplus t(xy)\mathring{S})=\oplus_{i=0}^{\infty}t(x)^{i}\mathring{T}.
\]
Thus $J=\oplus_{i=0}^{\infty}t(x)^{i}(\mathring{T}[x,y]x\oplus\mathring{T}[x,y]y\oplus\mathring{T}[x,y]^{2}\oplus\mathring{T}[x,y]xy)$,
and, regarding the relation between $J_{n}$ and $J_{n+1}$, we have
\begin{align*}
J_{n} & =t(x)^{n}\oplus_{i=0}^{\infty}t(x)^{i}(\mathring{T}[x,y]x\oplus\mathring{T}[x,y]y\oplus\mathring{T}[x,y]^{2}\oplus\mathring{T}[x,y]xy)\\
 & =\oplus_{i=n}^{\infty}t(x)^{i}(\mathring{T}[x,y]x\oplus\mathring{T}[x,y]y\oplus\mathring{T}[x,y]^{2}\oplus\mathring{T}[x,y]xy)\\
 & =t(x)^{n}(\mathring{T}[x,y]x\oplus\mathring{T}[x,y]y\oplus\mathring{T}[x,y]^{2}\oplus\mathring{T}[x,y]xy)\oplus J_{n+1}.
\end{align*}

\begin{cor}
\label{prop:UFD}Let $\hat{J}$ denote the closure of $J$ in the
degree topology induced by the grading of $R$. Then every $f\in\hat{J}$
can be uniquely written as 
\[
f=\sum_{n=0}^{\infty}\sum_{i=0}^{\infty}(U_{n,i}(f)+V_{n,i}(f)+W_{n,i}(f))
\]
where $U_{n,i}(f),V_{n,i}(f),W_{n,i}(f)$ are homogeneous of degree
$i$ of the form
\begin{align*}
U_{n,i}(f) & \in t(x)^{n}\cdot(\mathring{T}\cdot[x,y]x+\mathring{T}\cdot[x,y]y)\\
V_{n,i}(f) & \in t(x)^{n}\cdot\mathring{T}\cdot[x,y]^{2}\\
W_{n,i}(f) & \in t(x)^{n}\cdot\mathring{T}\cdot[x,y]xy.
\end{align*}
\end{cor}

\begin{rem}
Notice that given $f\in\hat{J}$, for any $n,i$ we either have $U_{n,i}(f)=0$
or $i=\deg(U_{n,i}(f))\geq n+3$. Similarly, we either have $i=\deg(V_{n,i}(f))\geq n+4$
or $V_{n,i}(f)=0$, and a similar property for $W_{n,i}(f)$.
\end{rem}

\begin{rem}
\label{rem:UFD}Notice that if $f$ lies in the closure of $J_{n_{0}}\subseteq\hat{J}$
for some $n_{0}$, then $U_{n,i}(f)=V_{n,i}(f)=W_{n,i}(f)=0$ for
any $n<n_{0}$ and $i\geq0$.
\end{rem}

\subsection{Some properties of the group $\hat{G}'$}

In this subsection we describe some properties of the group $\hat{G}'$
in view of the above discussion on the ring of the pseudo generic
matrices. As the ring $R=\left\langle x,\,y,\,1\cdot T\right\rangle =\oplus_{n=0}^{\infty}R^{(n)}$
is a graded ring (Corollary \ref{cor:graded}), we can define the
ring $\hat{R}$ of power series in $x,y$ and $1\cdot T$ with respect
to this grading on $R$. Then, the ideals
\[
\hat{R}_{n}=\left\{ \sum_{i=0}^{\infty}f_{i}\,|\,f_{i}\in R^{(i)},\,f_{0},...,f_{n-1}=0,\,\deg(f_{i})=i\right\} 
\]
\foreignlanguage{american}{serve as a basis of neighborhoods of zero
for the profinite topology of $\hat{R}$, making $\hat{R}$ a pro-$2$
ring.} Now, as for every $n\leq k$ we have $R^{(k)}\subseteq M_{2}(\Lambda,P^{n})$,
the topology of $R$ induced by the grading of $R$ is apriori stronger
than the topology of $R$ induced by the topology of $M_{2}(\Lambda)$.
Hence, we have a (continuous) ring homomorphism $\sigma:\hat{R}\to M_{2}(\Lambda)$
that sends the copy of $R$ in $\hat{R}$ to its copy in $M_{2}(\Lambda)$.
Hence, $\sigma$ induces also a group homomorphism from the closure
of $G=\left\langle 1+x,1+y\right\rangle $ in $\hat{R}$ to its closure
in $M_{2}(\Lambda)$, by sending the generators to their copy in $M_{2}(\Lambda)$.
Hence, we get the diagram
\[
\begin{array}{cccc}
\hat{F}=\left\langle \left\langle X,Y\right\rangle \right\rangle  & \rightarrow & \left\langle \left\langle 1+x,1+y\right\rangle \right\rangle \subseteq\hat{R}\\
 & \searrow & \downarrow\,\,\,\,\,\,\,\,\\
 &  & \left\langle \left\langle 1+x,1+y\right\rangle \right\rangle =\hat{G} & \subseteq M_{2}(\Lambda).
\end{array}
\]
Recall that by our assumption, the map $\pi|_{\hat{F}'}:\hat{F}'\overset{\sim}{\to}\hat{G}'$
defined by 
\begin{align*}
X & \mapsto1+x\in M_{2}(\Lambda)\,\,\,\,\,\,\,\,\,\,\,Y\mapsto1+y\in M_{2}(\Lambda).
\end{align*}
is an isomorphism. Hence, by the diagram we obtain that also the map
\begin{align*}
X & \mapsto1+x\in\hat{R}\,\,\,\,\,\,\,\,\,\,\,\,Y\mapsto1+y\in\hat{R}
\end{align*}
induces an isomorpism of $\hat{F}'$. Hence, we can replace $\hat{G}\subseteq M_{2}(\Lambda)$
by the closure of $G$ in $\hat{R}$. So, from now on, the notation
$\hat{G}$ refers to the closure of $G$ in $\hat{R}$, and the degree
of terms of elements in $\hat{G}$ will be determined by the grading
of $R$. In other words, we regard the elements of $R^{(n)}$ as the
elements of $R$ of degree $n$.
\begin{prop}
\label{prop:commutator}Let $g\in\hat{G}'$. Then $g=1+[x,y]\sum_{i=0}^{\infty}a_{i}$
for some $a_{i}\in R^{(i)}$. 
\end{prop}

\begin{proof}
Since
\begin{align*}
(1+[x,y]b_{1})(1+[x,y]b_{2}) & =1+[x,y](b_{1}+b_{2}+b_{1}[x,y]b_{2})\\
(1+[x,y]b)^{-1} & =1+[x,y](b\sum_{i=0}^{\infty}([x,y]b)^{i})
\end{align*}
it is enough to show that $\hat{G}'$ has a set of (topological) generators
of the form $1+[x,y]\sum_{i=0}^{\infty}a_{i}$. Now, the set of commutators
$\left\{ X^{n}Y^{m}X^{-n}Y^{-m}\,|\,n,m\in\mathbb{Z}\right\} $ generates
$F'$, the commutator subgroup of $F=\left\langle X,Y\right\rangle $
(cf. \cite{key-1}, Section 2.4, Problem 13). Hence, it is enough
to show that the elements of the set
\[
\left\{ (1+x)^{n}(1+y)^{m}(1+x)^{-n}(1+y)^{-m}\,|\,n,m\in\mathbb{Z}\right\} 
\]
are of the form $1+[x,y]\sum_{i=0}^{\infty}a_{i}$. Let $n,m\in\mathbb{Z}$.
Then
\begin{align*}
(1+x)^{n} & =1+x\cdot p_{1}(x)=1+x\cdot p_{1}(t(x))\\
(1+y)^{m} & =1+y\cdot p_{2}(y)=1+y\cdot p_{2}(t(y))
\end{align*}
where $p_{1}(x)$ and $p_{2}(y)$ are power series in $x$ and $y$,
respectively. Therefore
\begin{align*}
 & (1+x)^{n}(1+y)^{m}(1+x)^{-n}(1+y)^{-m}\\
 & =1+[x\cdot p_{1}(t(x)),y\cdot p_{2}(t(y))]\sum_{i,j=0}^{\infty}(x\cdot p_{1}(t(x))^{i}(y\cdot p_{2}(t(y)))^{j}\\
 & =1+[x,y]\cdot p_{1}(t(x))\cdot p_{2}(t(y))\sum_{i,j=0}^{\infty}(x\cdot p_{1}(t(x))^{i}(y\cdot p_{2}(t(y)))^{j}
\end{align*}
 has the needed form. 
\end{proof}
Recalling $\hat{J}$ from Corollary \ref{prop:UFD} we therefore have:
\begin{cor}
\label{cor:fix}Let $g=1+\sum_{i=1}^{\infty}a_{i}\in\hat{G}''$, $a_{i}\in R^{(i)}$,
where $\hat{G}''$ is the second derived subgroup of $G$. Then, we
have $a_{i}\in J$ for every $i$. In other words, we have $\hat{G}''\subseteq\hat{J}$.
\end{cor}

\begin{proof}
By definition, $\hat{G}''$ is generated (topologically) by commutators
of elements of $\hat{G}'$. Hence, as $J$ is an ideal of $R$ (Proposition
\ref{prop:J - f.g.}), it is enough to show that for every two elements
$h,k\in\hat{G}'$ the terms of $g=hkh^{-1}k^{-1}$ lie in $J$. By
Proposition \ref{prop:commutator} we can write $h=1+[x,y]\sum_{l=0}^{\infty}b_{l}$
and $k=1+[x,y]\sum_{l=0}^{\infty}c_{l}$ for some $b_{l},c_{l}\in R^{(l)}$.
Hence, for $g$ we have
\[
g=1+[[x,y]\sum_{l=0}^{\infty}b_{l},[x,y]\sum_{l=0}^{\infty}c_{l}]\cdot\sum_{i,j=0}^{\infty}([x,y]\sum_{l=0}^{\infty}b_{l})^{i}([x,y]\sum_{l=0}^{\infty}c_{l})^{j}.
\]
Now, by Part 2 of Lemma \ref{lem:a,b} we have
\begin{align*}
x[x,y] & =t(x)[x,y]+[x,y]x\,\,,\,\,\,\,\,\,y[x,y]=t(y)[x,y]+[x,y]y,
\end{align*}
so that $R[x,y]=[x,y]R$. Hence, one can write $g=1+[x,y]^{2}\cdot\sum_{l=0}^{\infty}d_{l}$
for some $d_{l}\in R^{(l)}$. As $[x,y]^{2}\in J$, and $J$ is an
ideal of $R$, the claim follows.
\end{proof}
We continue with the following lemma:
\begin{lem}
\label{lem:commuting}Let $\hat{F}$ be a free pro-$p$ group, and
$\hat{F}=\hat{F}_{1},\hat{F}_{2},\hat{F}_{3},...$ its lower central
series. If $1\neq g,h\in\hat{F}$ commute, then there exists $n$
such that $g,h\in\hat{F}_{n}-\hat{F}_{n+1}$. 
\end{lem}

\begin{proof}
Let $\hat{H}=\left\langle \left\langle g,h\right\rangle \right\rangle $
be the closure of the group generated by $g,h$ in $\hat{F}$. obviously,
$\hat{H}$ is abelian. In addition, by the well known result of Serre,
$\hat{H}$ is a free pro-$p$ group. It follows that $\hat{H}\cong\mathbb{Z}_{p}$
and that there exists an element $1\neq k\in\hat{H}$ such that $\hat{H}=\left\langle \left\langle k\right\rangle \right\rangle $.
As $\hat{F}$ is pro-nilpotent, there exists $n$ such that $k\in\hat{F}_{n}-\hat{F}_{n+1}$.
So $\hat{H}\subseteq\hat{F}_{n}$. Consider the homomorphic image
of $\hat{H}$ in the quotient $\Upsilon_{n}(\hat{F})=\hat{F}_{n}/\hat{F}_{n+1}$.
As $\Upsilon_{n}(\hat{F})$ is known to be isomorphic to $\mathbb{Z}_{p}^{l}$
for some $l$ \foreignlanguage{american}{(See \cite{key-12}, Proposition
2.7)}, the image of $\hat{H}$ in $\Upsilon_{n}(\hat{F})$ is a procyclic
pro-$p$ group of inifinte order. Hence, this image is isomorphic
to $\mathbb{Z}_{p}\cong\hat{H}$. As $\hat{H}$ is Hopfian, it follows
that the homomorphism of $\hat{H}$ into $\Upsilon_{n}(\hat{F})$
is an embedding. In particular, the image of $h,g$ in $\Upsilon_{n}(\hat{F})$
is non-trivial, so the claim follows.
\end{proof}
As we assume that $\hat{G}'=\left\langle \left\langle 1+x,1+y\right\rangle \right\rangle '$
is isomorphic to $\hat{F}'=\left\langle \left\langle X,Y\right\rangle \right\rangle '$
through the map $X\mapsto1+x$ and $Y\mapsto1+y$, we obtain:
\begin{cor}
\label{cor:commute}For every $1\neq g\in\hat{G}'$ and $m\geq1$,
we have $[(1+x)^{m},g]_{\hat{G}}\neq1$.
\end{cor}

\subsection{\label{subsec:Proof-of-Theorem}Proof of Theorem \ref{thm:main-1}}

We are now ready to present the core argument of the proof of Theorem
\ref{thm:main-1}. We begin with some definitions. In view of Corollaries
\ref{prop:UFD} and \ref{cor:fix} we define:
\begin{defn}
\label{def:minimal component}Let $1\neq g\in\hat{G}''$. We define
\begin{align*}
n(g) & =\min\{n\,|\,\exists i\in\mathbb{N}\,\,\textrm{such\,that}\,\,U_{n,i}(g)+V_{n,i}(g)+W_{n,i}(g)\neq0\}\\
\overline{n}(g) & =\min\{n\,|\,\exists i\in\mathbb{N}\,\,\textrm{such\,that}\,\,U_{n,i}(g)\neq0\}\\
\overline{i}(g) & =\min\{i\,|\,U_{\overline{n}(g),i}(g)\neq0\}.
\end{align*}
Notice that as $g\neq1$, $n(g)$ is well defined. If $U_{n,i}(g)=0$
for any $n,i\in\mathbb{N}$ we denote $\overline{n}(g)=\overline{i}(g)=\infty$.
Notice that by definition $n(g)\leq\overline{n}(g)\leq\overline{i}(g)$.
If $\overline{n}(g)<\infty$, we define 
\[
\textrm{min}_{x}(g)=U_{\overline{n}(g),\overline{i}(g)}(g).
\]
\end{defn}

\begin{defn}
Let $1\neq g\in\hat{G}''$. We define the operator 
\[
\varphi_{x}(g)=[(1+x)^{4},g]_{\hat{G}}=[1+t(x)^{3}x,g]_{\hat{G}}.
\]
When it will be convenient, we will use also the notation $g_{x}=\varphi_{x}(g)$. 
\end{defn}

\begin{defn}
We define $\breve{G}$ to be the subset of $\hat{G}''$ consisting
of all elements whose terms lie in $J_{3}$:
\[
\breve{G}=\left\{ g=1+\sum_{i=1}^{\infty}a_{i}\in\hat{G}''\,|\,a_{i}\in J_{3},\,\deg(a_{i})=i\right\} .
\]
Using the identity $(1+a)(1+b)(1+a)^{-1}=1+b+[a,b]\sum_{i=0}^{\infty}a^{i}$
and the fact that $J_{3}$ is a two sided ideal of $R$ (Corollary
\ref{cor:two sided}), it is easy to see that:
\end{defn}

\begin{prop}
The set $\breve{G}$ is a normal closed subgroup of $\hat{G}$. 
\end{prop}

In Subsection \ref{subsec:Preparations}, using Corollary \ref{cor:commute},
we prove that (see Lemma \ref{prop:pre-good}):
\begin{prop}
\label{prop:less than}For any $1\neq g\in\breve{G}$ we have $\overline{n}(g_{x})\leq n(g_{x})+1<\infty$.
\end{prop}

Recall the lower central series of $\hat{G}=\hat{G}_{1},\hat{G}_{2},...$.
In Subsection \ref{subsec:Conclusions}, using trace identities, the
Artin-Rees lemma and some ideas from the proof of Hilbert's basis
theorem, we prove the following main proposition:
\begin{prop}
\label{prop:step1}Let $1\neq g\in\breve{G}$ and $k\geq0$. There
exists a natural number $\rho$ such that if
\[
\overline{n}(g_{x})\geq\rho\,\,\,\,\textrm{and}\,\,\,\,\overline{i}(g_{x})\geq\rho+8k
\]
then there exists $h\in\breve{G}\cap\hat{G}_{k}$ that satisfies one
of the following conditions:
\begin{itemize}
\item $\overline{n}((gh)_{x})>\overline{n}(g_{x})$ or
\item $\overline{n}((gh)_{x})=\overline{n}(g_{x})$ and $\overline{i}((gh)_{x})>\overline{i}(g_{x})$.
\end{itemize}
\end{prop}

Assuming Proposition \ref{prop:step1} we now prove Theorem \ref{thm:main-1}.
Let $1\neq v\in\hat{G}''$. By Corollary \ref{cor:fix} all the terms
of $v$ lie in $J$. By Corollary \ref{cor:commute}, as $v\in\hat{G}''\subseteq\hat{G}'$,
it does not commute with $(1+x)^{2}$. Therefore, $[v,(1+x)^{2}]_{\hat{G}}\neq1$.
As $\hat{G}$ is pro nilpotent, there exists a unique $s$ such that
\[
[v,(1+x)^{2}]_{\hat{G}}\in\hat{G}_{s}-\hat{G}_{s+1}.
\]
By our assumption, $\hat{G}'\cong\hat{F}'=\left\langle \left\langle X,Y\right\rangle \right\rangle '$
through the map $X\mapsto1+x$ and $Y\mapsto1+y$. Hence, the quotient
$\hat{G}_{s}/\hat{G}_{s+1}$ is a free abelian pro-$2$ group \cite{key-12}.
Thus $([v,(1+x)^{2}]_{\hat{G}})^{2^{r}}\in\hat{G}_{s}-\hat{G}_{s+1}$
for every $r$. Recall $\rho$ from Proposition \ref{prop:step1},
and choose $r$ large enough such that $2^{r}\geq\rho+8(s+1)$. Define
\[
g=([v,(1+x)^{2}]_{\hat{G}})^{2^{r}}=([v,1+t(x)x]_{\hat{G}}){}^{2^{r}}\in\hat{G}_{s}-\hat{G}_{s+1}.
\]
As obviously, all the terms of $[v,(1+x)^{2}]_{\hat{G}}=[v,1+t(x)x]_{\hat{G}}$
lie in $t(x)J$, it follows that all the terms of $g$ lie in $t(x)^{2^{r}}J=J_{2^{r}}$.
In particular, $g\in\breve{G},$ and in view of Remark \ref{rem:UFD}
we have
\[
\overline{i}(g{}_{x})\geq\overline{n}(g{}_{x})\geq n(g_{x})\geq2^{r}\geq\rho+8(s+1)\geq\rho.
\]
({*}) Now, let $1\leq k_{1}\in\mathbb{N}$ be the largest number satisfying
the inequality
\[
\overline{i}(g_{x})\geq\rho+8(s+k_{1}).
\]
By Proposition \ref{prop:step1}, we can construct an element $h_{1}\in\breve{G}\cap G_{s+k_{1}}$
such that
\begin{itemize}
\item $\overline{n}((gh_{1})_{x})>\overline{n}(g_{x})$ or
\item $\overline{n}((gh_{1})_{x})=\overline{n}(g{}_{x})$ and $\overline{i}((gh_{1})_{x})>\overline{i}(g_{x})$.
\end{itemize}
Notice that as $g\in\breve{G}\cap(\hat{G}_{s}-\hat{G}_{s+1})$ and
$h_{1}\in\breve{G}\cap\hat{G}_{s+1}$ we have
\begin{align*}
gh_{1} & \in\breve{G}\cap(\hat{G}_{s}-\hat{G}_{s+1})\,\,,\,\,\,\,gh_{1}=g\,\,\textrm{mod}\,\,\hat{G}_{s+1}\\
\textrm{and}\,\,\,\,\overline{i}((gh_{1})_{x}) & \geq\overline{n}((gh_{1})_{x})\geq\overline{n}(g{}_{x})\geq\rho+8(s+1)\geq\rho.
\end{align*}
Hence, we can repeat the algorithm, starting from ({*}), for $gh_{1}$
instead of $g$. Repeating the algorithm, each time for $gh_{1}\cdot...\cdot h_{j}h_{j+1}$
instead of $gh_{1}\cdot...\cdot h_{j}$, we have two options:
\begin{itemize}
\item For some $j$ we have: $\overline{n}((gh_{1}\cdot...\cdot h_{j})_{x})>\overline{n}(g{}_{x})$.
The first time it happens we define $w_{1}=h_{1}\cdot...\cdot h_{j}$.
\item For every $j$ we have $\overline{n}((gh_{1}\cdot...\cdot h_{j})_{x})=\overline{n}(g{}_{x})$.
In this case 
\[
\overline{i}((gh_{1}\cdot...\cdot h_{j})_{x})\overset{j\to\infty}{\longrightarrow}\infty
\]
and therefore $k_{j}\overset{j\to\infty}{\longrightarrow}\infty$.
Hence, the sequence $\breve{G}\cap\hat{G}_{s+k_{j}}\ni h_{j}\overset{j\to\infty}{\longrightarrow}1$,
and therefore, the sequence $h_{1}\cdot...\cdot h_{j}$ converges
to an element $w_{1}\in\breve{G}\cap\hat{G}_{s+1}$ with the property
\[
\overline{n}((gw_{1})_{x})>\overline{n}(g{}_{x}).
\]
\end{itemize}
Now, as $g\in\breve{G}\cap(\hat{G}_{s}-\hat{G}_{s+1})$ and $w_{1}\in\breve{G}\cap\hat{G}_{s+1}$
we have
\begin{align*}
gw_{1} & \in\breve{G}\cap(\hat{G}_{s}-\hat{G}_{s+1})\,\,,\,\,\,\,gw_{1}=g\,\,\textrm{mod}\,\,\hat{G}_{s+1}\\
\textrm{and}\,\,\,\,\overline{i}((gw_{1})_{x}) & \geq\overline{n}((gw_{1})_{x})\geq\overline{n}(g{}_{x})+1\geq\rho+8(s+1)+1\geq\rho.
\end{align*}
Hence, we can repeat the algorithm, starting from ({*}), for $gw_{1}$
instead of $g$. Repeating the algorithm $7$ times more we get $w_{1},...,w_{8}\in\breve{G}\cap\hat{G}_{s+1}$
such that
\begin{align*}
gw_{1}\cdot...\cdot w_{8} & \in\breve{G}\cap(\hat{G}_{s}-\hat{G}_{s+1})\,\,,\,\,\,\,gw_{1}\cdot...\cdot w_{8}=g\,\,\textrm{mod}\,\,\hat{G}_{s+1}\\
\textrm{and}\,\,\,\,\overline{i}((gw_{1}\cdot...\cdot w_{8})_{x}) & \geq\overline{n}((gw_{1}\cdot...\cdot w_{8})_{x})\geq\overline{n}(g{}_{x})+8\geq\rho+8(s+2)\geq\rho.
\end{align*}
Hence, by repeating the algorithm, starting from ({*}), for $gw_{1}\cdot...\cdot w_{8}$
instead of $g$ we can construct $w_{9},...,w_{16}\in\breve{G}\cap\hat{G}_{s+2}$
such that 
\begin{align*}
gw_{1}\cdot...\cdot w_{16} & \in\breve{G}\cap(\hat{G}_{s}-\hat{G}_{s+1})\,\,,\,\,\,\,gw_{1}\cdot...\cdot w_{16}=g\,\,\textrm{mod}\,\,\hat{G}_{s+1}\\
\textrm{and}\,\,\,\,\overline{i}((gw_{1}\cdot...\cdot w_{16})_{x}) & \geq\overline{n}((gw_{1}\cdot...\cdot w_{16})_{x})\geq\overline{n}(g{}_{x})+16\geq\rho+8(s+3)\geq\rho.
\end{align*}
Repeating the algorithm, starting from ({*}), we can get a sequence
$w_{j}\overset{j\to\infty}{\longrightarrow}1$ such that the product
converges to an element
\[
w=\lim_{j\to\infty}w_{1}\cdot...\cdot w_{j}\in\breve{G}\cap\hat{G}_{s+1}
\]
with the properties $gw\in\breve{G}\cap(\hat{G}_{s}-\hat{G}_{s+1})$,
$gw=g\,\,\textrm{mod}\,\,\hat{G}_{s+1}$ and $\overline{n}((gw)_{x})=\infty$.
This is a contradiction to Proposition \ref{prop:less than}, and
hence, up to the proofs of Propositions \ref{prop:less than} and
\ref{prop:step1}, we have finished the proof of Theorem \ref{thm:main-1}. 

\subsection{\label{subsec:Preparations}Preparations for the proof of Proposition
\ref{prop:step1}}

In the following definition and discussion, given an element 
\[
a\in J_{n}=\overline{J}_{n}\oplus\overline{C}_{n}\oplus t(x)^{n}T\cdot[x,y]xy,
\]
we denote its projection to $\overline{J}_{n}$ by $\overline{a}$.
Regarding this notation, given an element $1\neq g=1+\sum_{i=1}^{\infty}a_{i}\in\hat{G}''$
where $a_{i}\in J$ is of degree $i$, during the following discussion,
it will be convenient to have in mind the following equivalent definitions
(compare with Definition (\ref{def:minimal component})): 
\begin{align*}
n(g) & =\max\{n\,|\,\forall i\,\,a_{i}\in J_{n}\}\\
\overline{n}(g) & =\max\{n\,|\,\forall i\,\,\overline{a}_{i}\in\overline{J}_{n}\}\\
\overline{i}(g) & =\min\{i\,|\,\overline{a}_{i}\notin\overline{J}_{\overline{n}(g)+1}\}.
\end{align*}

\begin{defn}
We say that an element $1\neq g\in\breve{G}$ is \uline{good} if
$\overline{n}(g)<\infty$ and it has the form
\begin{align*}
g=1+\textrm{min}_{x}(g) & +\textrm{terms\,\,in}\,\,\overline{J}_{\overline{n}(g)}\,\,\textrm{of\,\,degree}\,\,>\textrm{deg}(\textrm{min}_{x}(g))\\
 & +\textrm{terms\,\,in}\,\,\overline{J}_{\overline{n}(g)+1}\\
 & +\textrm{terms\,\,in}\,\,\overline{C}_{\overline{n}(g)-1}\\
 & +\textrm{terms\,\,in}\,\,J_{\overline{n}(g)+2}.
\end{align*}
\end{defn}

\begin{lem}
\label{prop:multi}Let $1\neq g,\,h\in\breve{G}$ be good elements
such that $n_{0}=\overline{n}(g)=\overline{n}(h)$, $i_{0}=\overline{i}(g)=\overline{i}(h)$
and $\textrm{min}_{x}(g)\neq\textrm{min}_{x}(h)$. Then $g\cdot h$
is also good, $\overline{n}(gh)=n_{0}$, $\overline{i}(gh)=i_{0}$,
and 
\[
\textrm{min}_{x}(g\cdot h)=\textrm{min}_{x}(g)+\textrm{min}_{x}(h).
\]
\end{lem}

\begin{proof}
Recall that by the definition of $\breve{G}$, the terms of $g,h$
are in $J_{3}\subseteq t(x)^{3}R$. Hence, $g\cdot h$ has the form
\begin{align}
g\cdot h=1+\textrm{min}_{x}(g)+\textrm{min}_{x}(h) & +\textrm{terms\,\,in}\,\,\overline{J}_{n_{0}}\,\,\textrm{of\,\,degree}\,\,>i_{0}\nonumber \\
 & +\textrm{terms\,\,in}\,\,\overline{J}_{n_{0}+1}\nonumber \\
 & +\textrm{terms\,\,in}\,\,\overline{C}_{n_{0}-1}\label{eq:gh1}\\
 & +\textrm{terms\,\,in}\,\,J_{n_{0}+2}.\nonumber 
\end{align}
By assumption, $0\neq\textrm{min}_{x}(g)+\textrm{min}_{x}(h)$ is
in $t(x)^{n_{0}}\cdot(\mathring{T}\cdot[x,y]x+\mathring{T}\cdot[x,y]y)$
and of degree $i_{0}$. By (\ref{eq:gh1}), it follows that $\overline{n}(g\cdot h)=n_{0}$,
$\bar{i}(g\cdot h)=i_{0}$, and that $\textrm{min}_{x}(g\cdot h)=\textrm{min}_{x}(g)+\textrm{min}_{x}(h)$.
Hence, $g\cdot h$ is good, as required.
\end{proof}
\begin{lem}
\label{prop:pre-good}Let $1\neq g=1+\sum_{i=1}^{\infty}a_{i}\in\breve{G}$.
Then:
\begin{itemize}
\item $g_{x}$ is good,
\item $n(g_{x})=n_{0}+3$ where $n_{0}=\max\{n\,|\,\forall i\,\,[x,a_{i}]\in J_{n}\}$,
and
\item $n(g_{x})\leq\overline{n}(g_{x})\leq n(g_{x})+1$.
\end{itemize}
\end{lem}

\begin{proof}
Notice first that as $g\neq1$ and $\breve{G}\subseteq\hat{G}''$,
by Corollary \ref{cor:commute}, $g$ does not commute with $1+x$,
and hence there exists $i$ such that $[x,a_{i}]\neq0$. Thus, $n_{0}$
is well defined. We claim now that $n(g_{x})=n_{0}+3$. Indeed, denote
$g_{x}=1+\sum_{i=1}^{\infty}b_{i}$, and write 
\begin{align*}
g_{x} & =1+\sum_{i=1}^{\infty}b_{i}=1+t(x)^{3}[x,\sum_{i=1}^{\infty}a_{i}]\sum_{k,l=0}^{\infty}(t(x)^{3}x)^{k}(\sum_{i=1}^{\infty}a_{i})^{l}.
\end{align*}
Now, by assumption, for every $i$ we have $a_{i}\in J_{3}\subseteq t(x)^{3}R$,
and $[x,a_{i}]\in J_{n_{0}}$. Hence, for every $i$ we have 
\begin{align*}
b_{i+4} & =t(x)^{3}[x,a_{i}]\,\,\,\,\textrm{mod}\,\,J_{n_{0}+3+3}.
\end{align*}
Hence $b_{i+4}\in J_{n_{0}+3}$ for every $i$, and thus $n(g_{x})\geq n_{0}+3$.
On the other hand, by assumption, there exists $i_{0}$ such that
$[x,a_{i_{0}}]\notin J_{n_{0}+1}$, so by Corollary \ref{prop:iff}
$t(x)^{3}[x,a_{i_{0}}]\notin J_{n_{0}+4}$. Therefore
\[
b_{i_{0}+4}=t(x)^{3}[x,a_{i_{0}}]+\textrm{terms\,\,in}\,\,J_{n_{0}+6}
\]
is not in $J_{n_{0}+4}$. Thus $n(g_{x})=n_{0}+3,$ which is the second
part of the assertion. 

Now, denote 
\[
a_{i}=\alpha_{i}[x,y]x+\beta_{i}[x,y]y+\gamma_{i}[x,y]^{2}+\delta_{i}[x,y]xy
\]
for $\alpha_{i},\,\beta_{i},\,\gamma_{i},\,\delta_{i}\in T$. By Lemmas
\ref{lem:a,b} and \ref{lem:x,y}, one has
\[
[x,a_{i}]=(t(x)\alpha_{i}+t(xy)\delta_{i})[x,y]x+t(x)\beta_{i}[x,y]y+\beta_{i}[x,y]^{2}.
\]
Set $\lambda_{i}=t(x)\alpha_{i}+t(xy)\delta_{i}$. Now, let $i_{0}$
be any index such that $b_{i_{0}+4}\notin J_{n(g_{x})+1}$. Then,
as $n_{0}+6=n(g_{x})+3$, we have
\begin{align*}
b_{i_{0}+4} & =t(x)^{3}[x,a_{i_{0}}]+\textrm{terms\,\,in}\,\,J_{n(g_{x})+3}\\
 & =t(x)^{3}(\lambda_{i_{0}}[x,y]x+t(x)\beta_{i_{0}}[x,y]y+\beta_{i_{0}}[x,y]^{2})+\textrm{terms\,\,in}\,\,J_{n(g_{x})+3}
\end{align*}
Hence, we have two options:
\begin{enumerate}
\item $t(x)^{3}\lambda_{i_{0}}\notin t(x)^{n(g_{x})+1}T$. In this case
we also have $\overline{b}_{i_{0}+4}\notin\overline{J}_{n(g_{x})+1}$.
In particular, $\overline{n}(g_{x})\leq n(g_{x})$. Otherwise:
\item $t(x)^{3}\beta_{i_{0}}\notin t(x)^{n(g_{x})+1}T$. Hence $t(x)^{4}\beta_{i_{0}}\notin t(x)^{n(g_{x})+2}T$.
So in this case we have $\overline{b}_{i_{0}+4}\notin\overline{J}_{n(g_{x})+2}$.
In particular ,$\overline{n}(g_{x})\leq n(g_{x})+1$. 
\end{enumerate}
In both cases we get $\overline{n}(g_{x})\leq n(g_{x})+1$, as claimed
in the third assertion. In particular, $\overline{n}(g_{x})<\infty$.
Hence $\textrm{min}_{x}(g_{x})$ exists and we have

\begin{align}
g_{x}= & 1+t(x)^{3}\sum_{i=0}^{\infty}(\lambda_{i_{0}}[x,y]x+t(x)\beta_{i_{0}}[x,y]y+\beta_{i_{0}}[x,y]^{2})+\textrm{terms\,\,in}\,\,J_{n(g_{x})+3}\nonumber \\
= & 1+\textrm{min}_{x}(g_{x})+\textrm{terms\,\,in}\,\,\overline{J}_{\overline{n}(g_{x})}\,\,\textrm{of\,\,degree}\,\,>\textrm{deg}(\textrm{min}_{x}(g_{x}))\nonumber \\
 & \qquad\qquad\quad\,\,\,\,+\textrm{terms\,\,in}\,\,\overline{J}_{\overline{n}(g_{x})+1}\label{eq:comp}\\
 & \qquad\qquad\quad\,\,\,\,+\textrm{terms\,\,in}\,\,\bar{C}_{n(g_{x})}\nonumber \\
 & \qquad\qquad\quad\,\,\,\,+\textrm{terms\,\,in}\,\,J_{n(g_{x})+3}.\nonumber 
\end{align}
Combining the latter estimation for $\overline{n}(g_{x})$ with equation
(\ref{eq:comp}) we get that $g_{x}$ is good, as required.
\end{proof}
\begin{lem}
\label{prop:t(x)}Let $1\neq g\in\breve{G}$ be a good element. Then
$\varphi_{x}(g)$ is also good, and 
\[
\textrm{min}_{x}(\varphi_{x}(g))=t(x)^{4}\textrm{min}_{x}(g)
\]
In particular $\overline{n}(\varphi_{x}(g))=\overline{n}(g)+4$.
\end{lem}

\begin{proof}
The fact that $\varphi_{x}(g)=g_{x}$ is good was already proven in
the previous lemma in a more general case. Now, write $g=1+a$ for
$a$ of the form
\begin{align*}
a=\textrm{min}_{x}(g) & +\textrm{terms\,\,in}\,\,\overline{J}_{\overline{n}(g)}\,\,\textrm{of\,\,degree}\,\,>\textrm{deg}(\textrm{min}_{x}(g))\\
 & +\textrm{terms\,\,in}\,\,\overline{J}_{\overline{n}(g)+1}\\
 & +\textrm{terms\,\,in}\,\,\overline{C}_{\overline{n}(g)-1}\\
 & +\textrm{terms\,\,in}\,\,J_{\overline{n}(g)+2}.
\end{align*}
 Recall that the terms of $g$ lie in $J_{3}\subseteq t(x)^{3}R$.
Hence
\begin{align*}
\varphi_{x}(g) & =1+t(x)^{3}[x,a]\sum_{k,l=0}^{\infty}(t(x)^{3}x){}^{k}a{}^{l}=1+t(x)^{3}[x,a]+\textrm{terms\,\,in}\,\,J_{\overline{n}(g)+6}.
\end{align*}
Noticing the identities
\begin{align*}
[x,[x,y]x] & =t(x)[x,y]x\,\,\,\,\textrm{and}\,\,\,\,[x,[x,y]y]=t(x)[x,y]y+[x,y]^{2}
\end{align*}
 we have:
\begin{itemize}
\item $[x,\textrm{min}_{x}(g)]=t(x)\textrm{min}_{x}(g)+\textrm{terms\,\,in}\,\,\overline{C}_{\overline{n}(g)}$.
\item $[x,\textrm{terms\,\,in}\,\,\overline{J}_{\overline{n}(g)}\,\,\textrm{of\,\,degree}\,\,>\textrm{deg}(\textrm{min}_{x}(g))]$
\begin{align*}
= & \textrm{terms\,\,in}\,\,\overline{J}_{\overline{n}(g)+1}\,\,\textrm{of\,\,degree}\,\,>\textrm{deg}(t(x)\textrm{min}_{x}(g))\\
 & +\textrm{terms\,\,in}\,\,\overline{C}_{\overline{n}(g)}
\end{align*}
\item $[x,\textrm{terms\,\,in}\,\,\overline{J}_{\overline{n}(g)+1}]=\textrm{terms\,\,in}\,\,\overline{J}_{\overline{n}(g)+2}+\textrm{terms\,\,in}\,\,\overline{C}_{\overline{n}(g)+1}$.
\item $[x,\textrm{terms\,\,in}\,\,\overline{C}_{\overline{n}(g)-1}]=0$.
\item $[x,\textrm{terms\,\,in}\,\,J_{\overline{n}(g)+2}]=\textrm{terms\,\,in}\,\,J_{\overline{n}(g)+2}$.
\end{itemize}
Hence
\begin{align*}
\varphi_{x}(g)= & 1+t(x)^{3}[x,a]+\textrm{terms\,\,in}\,\,J_{\overline{n}(g)+6}\\
= & 1+t(x)^{4}\textrm{min}_{x}(g)+\textrm{terms\,\,in}\,\,\overline{J}_{\overline{n}(g)+4}\,\,\textrm{of\,\,degree}\,\,>\textrm{deg}(t(x)^{4}\textrm{min}_{x}(g))\\
 & \qquad\qquad\qquad\quad\,\,+\textrm{terms\,\,in}\,\,\overline{J}_{\overline{n}(g)+5}\\
 & \qquad\qquad\qquad\quad\,\,+\textrm{terms\,\,in}\,\,\overline{C}_{\overline{n}(g)+3}\\
 & \qquad\qquad\qquad\quad\,\,+\textrm{terms\,\,in}\,\,J_{\overline{n}(g)+5}.
\end{align*}
Recalling Corollary \ref{prop:iff}, it follows that $\overline{n}(\varphi_{x}(g))=\overline{n}(g)+4$,
and that we have $\textrm{min}_{x}(\varphi_{x}(g))=t(x)^{4}\textrm{min}_{x}(g)$
as required.
\end{proof}
\begin{lem}
\label{prop:t(y)}Let $1\neq g\in\breve{G}$ be a good element, and
let $\varphi_{y}$ be the operator $\varphi_{y}(g)=[1+y,g]_{\hat{G}}$.
Then, $\varphi_{y}(g)$ is also good, and 
\[
\textrm{min}_{x}(\varphi_{y}(g))=t(y)\textrm{min}_{x}(g).
\]
In particular, $\overline{n}(\varphi_{y}(g))=\overline{n}(g)$.
\end{lem}

\begin{proof}
Write $g=1+a$ where $a$ is built up from
\begin{align*}
a=\textrm{min}_{x}(g) & +\textrm{terms\,\,in}\,\,\overline{J}_{\overline{n}(g)}\,\,\textrm{of\,\,degree}\,\,>\textrm{deg}(\textrm{min}_{x}(g))\\
 & +\textrm{terms\,\,in}\,\,\overline{J}_{\overline{n}(g)+1}\\
 & +\textrm{terms\,\,in}\,\,\overline{C}_{\overline{n}(g)-1}\\
 & +\textrm{terms\,\,in}\,\,J_{\overline{n}(g)+2}.
\end{align*}
In addition, we can think of $a$ as being built up of terms in $t(x)^{3}R$.
Hence
\begin{align*}
\varphi_{y}(g)= & 1+[y,a]\sum_{k,l=0}^{\infty}y{}^{k}a{}^{l}=1+[y,a]\sum_{k=0}^{\infty}y{}^{k}+\textrm{terms\,\,in}\,\,J_{\overline{n}(g)+3}.
\end{align*}

Using Part 4 of Lemma \ref{lem:a,b}, Corollary \ref{cor:trace=00003D0}
and the identities 
\[
[y,[x,y]x]=t(y)[x,y]x+[x,y]^{2}\,\,,\,\,\,\,[y,[x,y]y]=t(y)[x,y]y
\]
we have:
\begin{itemize}
\item $[y,\textrm{min}_{x}(g)]\sum_{k=0}^{\infty}y{}^{k}$
\begin{align*}
= & t(y)\textrm{min}_{x}(g)+\textrm{terms\,\,in}\,\,\overline{C}_{\overline{n}(g)}\\
 & +\textrm{terms\,\,in}\,\,\overline{J}_{\overline{n}(g)}\,\,\textrm{of\,\,degree}\,\,>\textrm{deg}(t(y)\textrm{min}_{x}(g))
\end{align*}
\item $[y,\textrm{terms\,\,in}\,\,\overline{J}_{\overline{n}(g)}\,\,\textrm{of\,\,degree}\,\,>\textrm{deg}(\textrm{min}_{x}(g))]\sum_{k=0}^{\infty}y{}^{k}$
\begin{align*}
= & \textrm{terms\,\,in}\,\,\overline{J}_{\overline{n}(g)}\,\,\textrm{of\,\,degree}\,\,>\textrm{deg}(t(y)\textrm{min}_{x}(g))\\
 & +\textrm{terms\,\,in}\,\,\overline{C}_{\overline{n}(g)}
\end{align*}
\item $[y,\textrm{terms\,\,in}\,\,\overline{J}_{\overline{n}(g)+1}]\sum_{k=0}^{\infty}y{}^{k}=\textrm{terms\,\,in}\,\,\overline{J}_{\overline{n}(g)+1}+\textrm{terms\,\,in}\,\,\overline{C}_{\overline{n}(g)+1}$.
\item $[y,\textrm{terms\,\,in}\,\,\overline{C}_{\overline{n}(g)-1}]\sum_{k=0}^{\infty}y{}^{k}=0$.
\item $[y,\textrm{terms\,\,in}\,\,J_{\overline{n}(g)+2}]\sum_{k=0}^{\infty}y{}^{k}=\textrm{terms\,\,in}\,\,J_{\overline{n}(g)+2}$.
\end{itemize}
Hence
\begin{align*}
\varphi_{y}(g)= & 1+[y,a]\sum_{k=0}^{\infty}y{}^{k}+\textrm{terms\,\,in}\,\,J_{\overline{n}(g)+3}\\
= & 1+t(y)\textrm{min}_{x}(g)+\textrm{terms\,\,in}\,\,\overline{J}_{\overline{n}(g)}\,\,\textrm{of\,\,degree}\,\,>\textrm{deg}(t(y)\textrm{min}_{x}(g))\\
 & \qquad\qquad\qquad\,\,\,\,\,+\textrm{terms\,\,in}\,\,\overline{J}_{\overline{n}(g)+1}\\
 & \qquad\qquad\qquad\,\,\,\,\,+\textrm{terms\,\,in}\,\,\overline{C}_{\overline{n}(g)}\\
 & \qquad\qquad\qquad\,\,\,\,\,+\textrm{terms\,\,in}\,\,J_{\overline{n}(g)+2}
\end{align*}
Recalling Corollary \ref{prop:iff}, we obtain the assertions in the
proposition.
\end{proof}
\begin{lem}
\label{prop:=00005Bx,y=00005D^4}Let $1\neq g\in\breve{G}$ be a good
element, and let $\psi$ be the operator $\psi(g)=[([1+x,1+y]_{\hat{G}})^{2},g]_{\hat{G}}$.
Then $\psi(g)$ is also good, and
\[
\textrm{min}_{x}(\psi(g))=[x,y]^{4}\textrm{min}_{x}(g)
\]
In particular, $\overline{n}(\psi(g))=\overline{n}(g)$.
\end{lem}

\begin{proof}
As in the previous lemma, we write $g=1+a$ where $a$ can be written
as
\begin{align*}
a=\textrm{min}_{x}(g) & +\textrm{terms\,\,in}\,\,\overline{J}_{\overline{n}(g)}\,\,\textrm{of\,\,degree}\,\,>\textrm{deg}(\textrm{min}_{x}(g))\\
 & +\textrm{terms\,\,in}\,\,\overline{J}_{\overline{n}(g)+1}\\
 & +\textrm{terms\,\,in}\,\,\overline{C}_{\overline{n}(g)-1}\\
 & +\textrm{terms\,\,in}\,\,J_{\overline{n}(g)+2}.
\end{align*}
In addition, we think of $a$ as being built up of terms in $t(x)^{3}R$. 

Denote $([1+x,1+y]_{\hat{G}})^{2}=1+b$. then
\begin{align*}
\psi(g)= & 1+[b,a]\sum_{k,l=0}^{\infty}b{}^{k}a{}^{l}=1+[b,a]\sum_{k=0}^{\infty}b{}^{k}+\textrm{terms\,\,in}\,\,J_{\overline{n}(g)+3}.
\end{align*}
Notice that for every $u,\,v\in M_{2}(\Lambda)$ we have $t(uv+vu)=t([u,v])=0$.
Also, notice that $([x,y]xy)^{2}=[x,y]^{3}xy$ and that $([x,y]x)^{2}=([x,y]y)^{2}=0$.
Hence, a direct computation shows that 
\begin{align*}
b=([x,y]\sum_{i,j=0}^{\infty}x^{i}y^{j})^{2}= & [x,y]^{2}+[x,y]^{3}xy\\
 & +\textrm{terms\,\,in\,\,}J\,\,\textrm{of\,\,zero\,\,trace\,\,of\,\,degree}\geq5\\
 & +\textrm{terms\,\,in}\,\,J\,\,\textrm{of\,\,degree}\,\,\geq9.
\end{align*}
Observe that by Part 1 of Lemma \ref{lem:a,b} and Corollary \ref{cor:trace=00003D0},
for $u,v\in J$ such that $t(u)=t(v)=0$ we have $[u,v]\in[x,y]^{2}T$.
Hence, using Part 4 of Lemma \ref{lem:a,b}, Corollary \ref{cor:trace=00003D0},
and the identities 
\[
[[x,y]xy,[x,y]x]=[x,y]^{2}\cdot[x,y]x\,\,,\,\,\,\,[[x,y]xy,[x,y]y]=[x,y]^{2}\cdot[x,y]y
\]
we have:
\begin{itemize}
\item $[b,\textrm{min}_{x}(g)]\sum_{k=0}^{\infty}b{}^{k}$
\begin{align*}
= & [b,\textrm{min}_{x}(g)]+\textrm{terms\,\,in}\,\,\overline{J}_{\overline{n}(g)}\,\,\textrm{of\,\,degree}\,\,>\textrm{deg}([x,y]^{4}\textrm{min}_{x}(g))\\
 & \qquad\qquad\,\,\,\,\,\,+\textrm{terms\,\,in}\,\,\overline{C}_{\overline{n}(g)}\\
= & [[x,y]^{3}xy,\textrm{min}_{x}(g)]\\
 & +\textrm{terms\,\,in}\,\,\overline{J}_{\bar{n}(g)}\,\,\textrm{of\,\,degree}\,\,>\textrm{deg}([x,y]^{4}\textrm{min}_{x}(g))\\
 & +\textrm{terms\,\,in}\,\,\overline{C}_{\overline{n}(g)}\\
= & [x,y]^{4}\textrm{min}_{x}(g)\\
 & +\textrm{terms\,\,in}\,\,\overline{J}_{\overline{n}(g)}\,\,\textrm{of\,\,degree}\,\,>\textrm{deg}([x,y]^{4}\textrm{min}_{x}(g))\\
 & +\textrm{terms\,\,in}\,\,\overline{C}_{\overline{n}(g)}
\end{align*}
\item $[b,\textrm{terms\,\,in}\,\,\overline{J}_{\overline{n}(g)}\,\,\textrm{of\,\,degree}\,\,>\textrm{deg}(\textrm{min}_{x}(g))]\sum_{k=0}^{\infty}b{}^{k}$
\begin{align*}
= & [b,\textrm{terms\,\,in}\,\,\overline{J}_{\overline{n}(g)}\,\,\textrm{of\,\,degree}\,\,>\textrm{deg}(\textrm{min}_{x}(g))]\\
 & +\textrm{terms\,\,in}\,\,\overline{J}_{\overline{n}(g)}\,\,\textrm{of\,\,degree}\,\,>\textrm{deg}([x,y]^{4}\textrm{min}_{x}(g))\\
 & +\textrm{terms\,\,in}\,\,\overline{C}_{\overline{n}(g)}\\
= & \textrm{terms\,\,in}\,\,\overline{J}_{\overline{n}(g)}\,\,\textrm{of\,\,degree}\,\,>\textrm{deg}([x,y]^{4}\textrm{min}_{x}(g))\\
 & +\textrm{terms\,\,in}\,\,\overline{C}_{\overline{n}(g)}
\end{align*}
\item $[b,\textrm{terms\,\,in}\,\,\overline{J}_{\overline{n}(g)+1}]\sum_{k=0}^{\infty}b{}^{k}=\textrm{terms\,\,in}\,\,\overline{J}_{\overline{n}(g)+1}+\textrm{terms\,\,in}\,\,\overline{C}_{\overline{n}(g)+1}$.
\item $[b,\textrm{terms\,\,in}\,\,\overline{C}_{\overline{n}(g)}]\sum_{k=0}^{\infty}b{}^{k}=0$.
\item $[b,\textrm{terms\,\,in}\,\,J_{\overline{n}(g)+2}]\sum_{k=0}^{\infty}b{}^{k}=\textrm{terms\,\,in}\,\,J_{\overline{n}(g)+2}$.
\end{itemize}
Hence
\begin{align*}
\psi(g)= & 1+[b,a]\sum_{k=0}^{\infty}b{}^{k}+\textrm{terms\,\,in}\,\,J_{\overline{n}(g)+3}\\
= & 1+[x,y]^{4}\textrm{min}_{x}(g)+\textrm{terms\,\,in}\,\,\overline{J}_{\overline{n}(g)}\,\,\textrm{of\,\,degree}\,\,>\textrm{deg}([x,y]^{4}\textrm{min}_{x}(g))\\
 & \qquad\qquad\qquad\quad\,\,\,\,+\textrm{terms\,\,in}\,\,\overline{J}_{\overline{n}(g)+1}\\
 & \qquad\qquad\qquad\quad\,\,\,\,+\textrm{terms\,\,in}\,\,\overline{C}_{\overline{n}(g)}\\
 & \qquad\qquad\qquad\quad\,\,\,\,+\textrm{terms\,\,in}\,\,J_{\overline{n}(g)+2}.
\end{align*}
Recalling Corollary \ref{prop:iff}, we obtain the assertions in the
proposition.
\end{proof}

\subsection{\label{subsec:Conclusions}Proof of Proposition \ref{prop:step1}}

In this subsection we are going to prove Proposition \ref{prop:step1}.
Let us recall the proposition.
\begin{prop}
Let $1\neq g\in\breve{G}$ and $k\geq0$. There exists a natural number
$\rho$ such that if
\[
\overline{n}(g_{x})\geq\rho\,\,\,\,\textrm{and}\,\,\,\,\overline{i}(g_{x})\geq\rho+8k
\]
then there exists $h\in\breve{G}\cap\hat{G}_{k}$ that satisfies one
of the following conditions:
\begin{itemize}
\item $\overline{n}((gh)_{x})>\overline{n}(g_{x})$ or
\item $\overline{n}((gh)_{x})=\overline{n}(g_{x})$ and $\overline{i}((gh)_{x})>\overline{i}(g_{x})$.
\end{itemize}
\end{prop}

\begin{proof}
Denote the subring (with identity) $\breve{S}=\left\langle t(x)^{4},\,t(y),\,[x,y]^{4}\right\rangle \subseteq S$,
and the ideal $\breve{I}=t(x)^{4}\breve{S}\vartriangleleft\breve{S}$.
It is easy to see that $S$ is a finitely generated $\breve{S}$-module,
and therefore $J$ is a finitely generated $\breve{S}$-module (see
Propositions \ref{prop:ST} and \ref{prop:J - f.g.}). Define the
set 
\[
M=\left\{ \textrm{min}_{x}(g)\,|\,g\in\breve{G}\,\,\textrm{is\,\,good}\right\} 
\]
and define $N$ to be the $\breve{S}$-submodule of $J$ generated
by the elements of $M$. As $\breve{S}$ is Noetherian and $J$ is
a finitely generated $\breve{S}$-module, $N$ is also a finitely
generated $\breve{S}$-module, and hence there exists a finite subset
$\Theta\subseteq\breve{G}$ of good elements such that
\[
N=\sum_{\theta\in\Theta}\breve{S}\cdot\textrm{min}_{x}(\theta).
\]
In addition, by the Artin-Rees lemma, there exists a number $\varrho$
such that
\[
M\cap(J_{4\varrho})\subseteq N\cap(J\cdot\breve{I}^{\varrho})\subseteq N\cdot\breve{I}^{2}.
\]
\begin{defn}
We define 
\[
\rho=\max\{4\varrho,\deg(\textrm{min}_{x}(\theta))+8\,|\,\theta\in\Theta\}.
\]
\end{defn}

Now, let $1\neq g\in\breve{G}$ and $k\geq0$ as in the proposition.
By assumption $\bar{n}(g_{x})\geq4\varrho$. By Lemma \ref{prop:pre-good},
$g_{x}\in\breve{G}$ is good. Hence
\[
\textrm{min}_{x}(g_{x})\in M\cap J_{4\rho}\subseteq N\cdot\breve{I}^{2}.
\]
Hence, there exist good elements $\mathring{h}_{1},...,\mathring{h}_{m}\in\Theta$
(repetitions are allowed) and elements $\lambda_{1},...,\lambda_{m}\in\breve{S}$
such that 
\[
\textrm{min}_{x}(g_{x})=t(x)^{8}(\lambda_{1}\cdot\textrm{min}_{x}(\mathring{h}_{1})+...+\lambda_{m}\cdot\textrm{min}_{x}(\mathring{h}_{m}))
\]
where $\lambda_{j}=t(x)^{4u_{j}}t(y)^{v_{j}}[x,y]^{4w_{j}}$, $j=1,...,m$,
for some $u_{j},\,v_{j},\,w_{j}\geq0$. Also, as each of the summands
is homogeneous, we can assume that each summand $\lambda_{j}\cdot\textrm{min}_{x}(\mathring{h}_{j})$
satisfies
\[
\deg(\lambda_{j}\cdot\textrm{min}_{x}(\mathring{h}_{j}))+8=\deg(\textrm{min}_{x}(g_{x}))=\bar{i}(g_{x})\geq\rho+8k.
\]
As $\rho\geq\deg(\textrm{min}_{x}(\theta))+8$ for any $\theta\in\Theta$,
we have $\deg(\lambda_{j})\geq8k$, and hence $u_{j}+v_{j}+w_{j}\geq k$
for any $j$. Now, define
\[
h_{j}=\psi^{w_{j}}\circ\varphi_{y}^{v_{j}}\circ\varphi_{x}^{u_{j}}(\mathring{h}_{j})\in\breve{G}\cap G_{k}.
\]
By Lemmas \ref{prop:t(x)}, \ref{prop:t(y)}, and \ref{prop:=00005Bx,y=00005D^4}
we have 
\[
\textrm{min}_{x}(h_{j})=t(x)^{4u_{j}}t(y)^{v_{j}}[x,y]^{4w_{j}}\textrm{min}_{x}(\mathring{h}_{j})=\lambda_{j}\cdot\textrm{min}_{x}(\mathring{h}_{j}).
\]
Therefore, denoting $\Omega=\left\{ j\,|\,\overline{n}(g{}_{x})=\overline{n}(h_{j})+8,\,\overline{i}(g{}_{x})=\overline{i}(h_{j})+8\right\} $,
by Corollary \ref{prop:UFD} we have 
\begin{equation}
\textrm{min}_{x}(g_{x})=t(x)^{8}\sum_{j=1}^{m}\textrm{min}_{x}(h_{j})=t(x)^{8}\sum_{j\in\Omega}\textrm{min}_{x}(h_{j})\label{eq:sum}
\end{equation}
Hence, we can assume that $\overline{n}(g{}_{x})=\overline{n}(h_{j})+8$,
$\overline{i}(g{}_{x})=\overline{i}(h_{j})+8$ for every $j$ and
in particular $\overline{n}(h_{j})=\overline{n}(h_{j'})$ and $\overline{i}(h_{j})=\overline{i}(h_{j'})$
for every $j,j'$. In addition, without loss of generality, we can
assume that for every $1\leq l\leq m-1$ one has
\[
\sum_{j=1}^{l}\textrm{min}_{x}(h_{j})\neq\textrm{min}_{x}(h_{l+1}).
\]
Otherwise, if $\sum_{j=1}^{l}\textrm{min}_{x}(h_{j})=\textrm{min}_{x}(h_{l+1})$
for some $l$, we can omit $h_{1},...,h_{l+1}$ from the sum in equation
(\ref{eq:sum}). 
\begin{defn}
Under the above assumptions we define 
\[
h=\varphi_{x}(\prod_{j=1}^{m}h_{j})\in\breve{G}\cap G_{k}.
\]
\end{defn}

It remains to show that $h$ satisfies the conditions in Proposition
\ref{prop:step1}. Indeed, by Lemmas \ref{prop:multi} and \ref{prop:t(x)},
$h$ is good, 
\[
\textrm{min}_{x}(h)=t(x)^{4}\sum_{j=1}^{m}\textrm{min}_{x}(h_{j}),
\]
and we have $\overline{n}(g_{x})=\overline{n}(h_{x})=\overline{n}(h)+4$.
Write $g=1+a=1+\sum_{i=1}^{\infty}a_{i}$ and $h=1+b=1+\sum_{i=1}^{\infty}b_{i}$,
so 
\[
gh=1+a+b+ab.
\]
By the construction, $h$ is of the form $h=\varphi_{x}(h')$ for
some good element $h'\in\breve{G}$. Hence $\overline{n}(h)-1\leq n(h)$
by Lemma \ref{prop:pre-good}. Thus, as the terms of $a$ lie in $t(x)^{3}R$
we have
\[
ab=\textrm{terms\,\,in}\,\,J_{n(h)+3}\subseteq J_{\overline{n}(h)+2}.
\]
By Lemma \ref{prop:pre-good} we have 
\begin{align*}
\max\{n\,|\,\forall i\,\,[x,a_{i}]\in J_{n}\} & =n(g_{x})-3\geq\overline{n}(g_{x})-4=\overline{n}(h)\\
\max\{n\,|\,\forall i\,\,[x,b_{i}]\in J_{n}\} & =n(h_{x})-3\geq\overline{n}(h_{x})-4=\overline{n}(h).
\end{align*}
 Therefore
\begin{align*}
(gh)_{x}= & 1+t(x)^{3}[x,a+b+ab]\sum_{i,j=0}^{\infty}(t(x)^{3}x)^{i}(a+b+ab)^{j}\\
= & 1+t(x)^{3}[x,a]+t(x)^{3}[x,b]+\textrm{terms\,\,in}\,\,J_{\overline{n}(h)+5}.
\end{align*}
Now, as in the proof of Lemma \ref{prop:pre-good}, we have
\begin{align*}
t(x)^{3}[x,a]=\textrm{min}_{x}(g_{x}) & +\textrm{terms\,\,in}\,\,\overline{J}_{\overline{n}(g_{x})}\,\,\textrm{of\,\,degree}\,\,>\textrm{deg}(\textrm{min}_{x}(g_{x}))\\
 & +\textrm{terms\,\,in}\,\,\overline{J}_{\overline{n}(g_{x})+1}+\textrm{terms\,\,in}\,\,\overline{C}_{\overline{n}(g_{x})-1}
\end{align*}
\begin{align*}
t(x)^{3}[x,b]=\textrm{min}_{x}(h_{x}) & +\textrm{terms\,\,in}\,\,\overline{J}_{\overline{n}(h_{x})}\,\,\textrm{of\,\,degree}\,\,>\textrm{deg}(\textrm{min}_{x}(h_{x}))\\
 & +\textrm{terms\,\,in}\,\,\overline{J}_{\overline{n}(h_{x})+1}+\textrm{terms\,\,in}\,\,\overline{C}_{\overline{n}(h_{x})-1}.
\end{align*}
Now, as $h$ is good, we have 
\[
\textrm{min}_{x}(h_{x})=t(x)^{4}\textrm{min}_{x}(h)=\textrm{min}_{x}(g_{x})
\]
by the construction of $h$ and Lemma \ref{prop:t(x)}. In addition,
$\overline{n}(g_{x})=\overline{n}(h)+4=\overline{n}(h_{x})$. Hence,
we get that
\begin{align*}
(gh)_{x}= & 1+t(x)^{3}[x,a]+t(x)^{3}[x,b]+\textrm{terms\,\,in}\,\,J_{\overline{n}(g_{x})+1}\\
= & 1+\textrm{terms\,\,in}\,\,\overline{J}_{\overline{n}(g_{x})}\,\,\textrm{of\,\,degree}\,\,>\textrm{deg}(\textrm{min}_{x}(g_{x}))\\
 & +\textrm{terms\,\,in}\,\,\overline{C}_{\overline{n}(g_{x})-1}\\
 & +\textrm{terms\,\,in}\,\,J_{\overline{n}(g_{x})+1}.
\end{align*}
It follows that either $\textrm{min}_{x}((gh)_{x})\in\overline{J}_{\overline{n}(g_{x})+1}$,
or $\textrm{min}_{x}((gh)_{x})\in\overline{J}_{\overline{n}(g_{x})}$
and $\textrm{deg}(\textrm{min}_{x}((gh)_{x}))>\textrm{deg}(\textrm{min}_{x}(g_{x}))$,
which is equivalent to our assertion. 
\end{proof}

\section{\label{sec:char=00003D0}Some remarks regarding $p=2$ versus $p\protect\neq2$}

\subsection{\label{subsec:universal 1}Stating the dichotomy}

\selectlanguage{american}%
In this section we describe exactly where Zubkov's approach fails
when $p=2$, and that in some sense, $2\times2$ pro-$2$ linear groups
indeed have less pro-$2$ identities (if any). We recall the notation
from $\mathsection$\ref{sec:Zubkov}:
\begin{itemize}
\item $\Pi_{*}=\mathbb{Z}_{p}\left\langle \left\langle x_{i,j},y_{i,j}\,|\,1\leq i,j\leq2\right\rangle \right\rangle $
is the associative ring (with identity) of formal power series in
the free commuting variables $x_{i,j}$ and $y_{i,j}$ over $\mathbb{Z}_{p}$.
\item $\Pi_{*}\vartriangleright Q_{*n}=\left\{ f=\sum_{i=0}^{\infty}f_{i}\in\Pi_{*}\,|\,f_{0},...,f_{n-1}=0,\,\deg(f_{i})=i\right\} $.
\item $x_{*},y_{*}\in M_{d}(\Pi_{*},Q_{*1})=\ker(M_{d}(\Pi_{*})\to M_{d}(\Pi_{*}/Q_{*1}))$
are the generic matrices
\[
x_{*}=\left(\begin{array}{cc}
x_{11} & x_{12}\\
x_{21} & x_{22}
\end{array}\right),\,\,\,\,\,\,y_{*}=\left(\begin{array}{cc}
y_{11} & y_{12}\\
y_{21} & y_{22}
\end{array}\right).
\]
\item $\pi_{*}:\hat{F}\to\hat{G}_{*}=\left\langle \left\langle 1+x_{*},1+y_{*}\right\rangle \right\rangle \subseteq1+M_{d}(\Pi_{*},Q_{*1})$
is the \foreignlanguage{english}{universal representation}.
\end{itemize}
\selectlanguage{english}%
Going back to Theorem \ref{thm:universal-1} we get that in some sense,
measuring how far $\hat{G}_{*}$ is from being a free pro-$p$ group
measures the amount of pro-$p$ identities of $2\times2$ pro-$p$
linear groups. 

\selectlanguage{american}%
Now, let $\hat{H}$ be a pro-$p$ group generated by two elements.
Let $\hat{H}=\hat{H}_{1},\hat{H}_{2},...$ be the lower central series
of $\hat{H}$, and denote $\Upsilon_{n}(\hat{H})=\hat{H}_{n}/\hat{H}_{n+1}$.
Notice that from the definition of the lower central series, for every
$n$ we have an epimorphism $\hat{F}_{n}\twoheadrightarrow\hat{H}_{n}$
that induces an epimorphism
\[
\Upsilon_{n}(\hat{F})\twoheadrightarrow\Upsilon_{n}(\hat{H}).
\]
\foreignlanguage{english}{The following proposition suggests that
one way to measure how far $\hat{H}$ is from being a free pro-$p$
group is to evaluate $\Upsilon_{n}(\hat{H})$.}
\begin{prop}
\label{prop:gap}The map $\hat{F}\twoheadrightarrow\hat{H}$ is an
isomorphiam if and only if for every $n$ the abelian pro-$p$ group
$\Upsilon_{n}(\hat{H})$ is isomorphic to $\Upsilon_{n}(\hat{H})\cong\Upsilon_{n}(\hat{F})\cong\mathbb{Z}_{p}^{l_{2}(n)}$
(see equation (\ref{eq:l2n}) for the definition of $l_{2}(n)$).
\end{prop}

\selectlanguage{english}%
\begin{proof}
If $\hat{F}\twoheadrightarrow\hat{H}$ is an isomorphism then clearly
$\Upsilon_{n}(\hat{H})\cong\Upsilon_{n}(\hat{F})\cong\mathbb{Z}_{p}^{l_{2}(n)}$.
On the other hand, if $\Upsilon_{n}(\hat{H})\cong\Upsilon_{n}(\hat{F})\cong\mathbb{Z}_{p}^{l_{2}(n)}$
for every $n$, then as\foreignlanguage{american}{ $\mathbb{Z}_{p}^{l_{2}(n)}$
is Hopfian (as a finitely generated profinite group), we get that
the surjective map $\Upsilon_{n}(\hat{F})\twoheadrightarrow\Upsilon_{n}(\hat{H})$}
is an isomorphism for every $n$. Now, assume that $\hat{F}\twoheadrightarrow\hat{H}$
is not an isomorphism, and let $1\neq g\in\ker(\hat{F}\twoheadrightarrow\hat{H})$.
As $\hat{F}$ is pro nilpotent, there exists a unique $n$ such that
$g\in\hat{F}_{n}-\hat{F}_{n+1}$. For this $n$ the map $\Upsilon_{n}(\hat{F})\twoheadrightarrow\Upsilon_{n}(\hat{H})$
is not injective, what leads to a contradiction. 
\end{proof}
We are going to prove the following proposition:
\begin{prop}
\label{prop:dichotomy}For every $p$ and every $n\leq5$ one has
$\Upsilon_{n}(\hat{G}_{*})\cong\mathbb{Z}_{p}^{l_{2}(n)}$. Continuing
to $n=6$ we have:
\begin{itemize}
\item For $p\neq2$ (Zubkov): $\Upsilon_{6}(\hat{G}_{*})\cong\mathbb{Z}_{p}^{6}$,
and in general $\Upsilon_{n}(\hat{G}_{*})\cong\mathbb{Z}_{p}^{m(n)}$
where
\[
m(n)=\begin{cases}
n(n+2)/8 & n\,\,\,is\,\,\,even\\
(n-1)(n+1)/4 & n\,\,\,is\,\,\,odd.
\end{cases}
\]
\item For $p=2$: $\Upsilon_{6}(\hat{G}_{*})$ is an abelian group that
is generated by at least $l_{2}(6)=9$ generators.
\end{itemize}
\end{prop}

In view of Proposition \ref{prop:gap} and Theorem \ref{thm:universal-1},
Proposition \ref{prop:dichotomy} shows that in some sense $\hat{G}_{*}$
is closer to being a free pro-$p$ group when $p=2$, and that there
is a real difference between the case $p=2$ and the cases $p\neq2$.
We notice that we do not have a useful tool to say substantially more
informative details regarding this difference. We cannot even say
that $\Upsilon_{6}(\hat{G}_{*})\cong\mathbb{Z}_{p}^{9}$ when $p=2$,
i.e. that $\Upsilon_{6}(\hat{G}_{*})$ is torsion free. However, on
the way to proving Proposition \ref{prop:dichotomy}, we will show
how this dichotomy arises in more detail than is actually needed in
order to prove the proposition (see Proposition \ref{prop:minimal}). 

\subsection{\label{subsec:Pseudo-Generic}The Pseudo Generic Matrices}

We start with presenting appropriate ``pseudo generic matrices''
that will help us to present Zubkov's approach in a bit simpler way
than in \cite{key-3}. We notice that we can define these ``pseudo
generic matrices'' in a similar way as we did in Section \ref{sec:char=00003D2}.
However, it turns out that over $\mathbb{Z}_{p}$, in order to move
from the original generic matrices to the pseudo generic matrices,
one can use a simpler argument that allows us to use a much simpler
definition for the pseudo generic matrices (see Corollary \ref{cor:min-new}). 

Let $x$ and $y$ denote the \uline{pseudo generic} matrices \foreignlanguage{american}{
\[
x=\left(\begin{array}{cc}
x_{11} & x_{12}\\
0 & 0
\end{array}\right),\,\,\,\,\,\,y=\left(\begin{array}{cc}
y_{11} & 0\\
y_{21} & 0
\end{array}\right)\in M_{2}(\Pi)
\]
where $\Pi=\mathbb{Z}_{p}\left\langle \left\langle x_{12},x_{11},y_{11},y_{21}\right\rangle \right\rangle \subseteq\Pi_{*}$.
Notice that $\det(x)=\det(y)=0$. }Being careful to distinguish between
$+$ and $-$, the following lemmas are obtained in a similar way
as Lemmas \ref{lem:a,b} and \ref{lem:x,y}:
\begin{lem}
\label{lem:a,b-1}Let $a,b\in M_{2}(\Pi)$. Then:
\begin{enumerate}
\item $ab+ba=t(a)b+t(b)a+\left(t(ab)-t(a)t(b)\right)\cdot1$.
\item $[a,b,a]=-t(a)[a,b]+2[a,b]a$. 
\item If $t(a)=0$, then $a^{2}=-\det(a)\cdot1\in M_{2}(\Pi)$ is central.
\end{enumerate}
\end{lem}

\begin{lem}
\label{lem:x,y-1}For the \foreignlanguage{american}{pseudo-generic
matrices} $x,\,y$ we have
\begin{enumerate}
\item $x^{2}=t(x)x,\,\,\,\,y^{2}=t(y)y,\,\,\,\,(xy)^{2}=t(xy)xy,\,\,\,\,(yx)^{2}=t(xy)yx$.
\item $xyx=t(xy)x,\,\,\,\,yxy=t(xy)y$.
\item $[x,y]^{2}=(t(xy)^{2}-t(x)t(y)t(xy))\cdot1$.
\end{enumerate}
\end{lem}

We define the subrings (with identity) of $\Pi$ 
\begin{align*}
T & =\mathbb{Z}_{p}\left\langle t(x),\,t(y),\,t(xy)\right\rangle \subseteq\Pi\\
S & =\mathbb{Z}_{p}\left\langle t(x)^{2},\,t(y)^{2},\,[x,y]^{2}\right\rangle \subseteq\Pi.
\end{align*}

The following proposition is proved by similar (but easier) arguments
as Proposition \ref{prop:free}.
\begin{prop}
\label{prop:free-1}The ring $T$ is freely generated by $t(x),\,t(y),\,t(xy)$
as a commutative algebra over $\mathbb{Z}_{p}$. 
\end{prop}

As a corollary of Part 3 in Lemma \ref{lem:x,y-1} and the above proposition
we have:
\begin{cor}
\label{prop:ST-1}The ring $S$ is contained in $T$, and it is freely
generated by $t(x)^{2},\,t(y)^{2},\,[x,y]^{2}$ as a commutative algebra
over $\mathbb{Z}_{p}$. 
\end{cor}

\subsection{\label{subsec:structure}The structure of a minimal component}

\selectlanguage{american}%
Noticing that $\Pi$ can also be viewed as the image of $\Pi_{*}$
under the projection 
\[
x_{21},x_{22},y_{12},y_{22}\overset{\sigma}{\mapsto}0,
\]
we use the following notation regarding the pseudo generic matrices
(see Subsection \ref{subsec:structure-1}):
\selectlanguage{english}%
\begin{itemize}
\item $\mathring{\Pi}=\sigma(\mathring{\Pi}_{*})=\mathbb{Q}_{p}\left\langle \left\langle x_{12},x_{11},y_{11},y_{21}\right\rangle \right\rangle $\foreignlanguage{american}{. }
\item $\Pi=\sigma(\Pi_{*})=\mathbb{Z}_{p}\left\langle \left\langle x_{12},x_{11},y_{11},y_{21}\right\rangle \right\rangle \subseteq\mathring{\Pi}$\foreignlanguage{american}{. }
\selectlanguage{american}%
\item $Q_{n}=\sigma(Q_{*n})$ and $\mathring{Q}_{n}=\sigma(\mathring{Q}_{*n})$.
\item $A=\sigma(A_{*})=\mathbb{Z}_{p}\left\langle \left\langle x,y\right\rangle \right\rangle \subseteq M_{2}(\Pi)$.
\item $\mathring{L}=\sigma(\mathring{L}_{*})$ = the Lie algebra generated
by $x,y$ over $\mathbb{Q}_{p}$.
\item $\mathring{L}^{(n)}=\sigma(\mathring{L}_{*}^{(n)})$ = the subspace
of $L$ of homogeneous elements of degree $n$.
\item $L=\sigma(L_{*})\subseteq\mathring{L}$ = the Lie algebra generated
by $x,y$ over $\mathbb{Z}_{p}$.
\item $L^{(n)}=\sigma(L_{*}^{(n)})\subseteq\mathring{L}^{(n)}$ = the additive
subgroup of $L$ of homogeneous elements of degree $n$.
\item For an element of the form $g=1+a_{n}+a_{n+1}+...\in1+M_{2}(\Pi,Q_{1})$
\foreignlanguage{english}{where $a_{i}$ is the term of $g$ of degree
$i$, and $a_{n}\neq0$, we denote
\[
\min(g)=a_{n}.
\]
}
\end{itemize}
By applying $\sigma$ on Zubkov's formulas from\foreignlanguage{american}{
$\mathsection$\ref{sec:Zubkov}}, one can see now that regarding
the pseudo generic matrices, one has\foreignlanguage{american}{
\begin{align*}
[x,y,x,x] & =\alpha[x,y]\,\,\,\,\textrm{for}\,\,\,\,\alpha=t(x)^{2}\\{}
[x,y,x,y]=[x,y,y,x] & =\beta[x,y]\,\,\,\,\textrm{for}\,\,\,\,\beta=2t(xy)-t(x)t(y)\\{}
[x,y,y,y] & =\gamma[x,y]\,\,\,\,\textrm{for}\,\,\,\,\gamma=t(y)^{2}
\end{align*}
and hence
\begin{align}
\mathring{L}^{(n)} & =\begin{cases}
\sum_{r+s+t=(n-2)/2}(\mathbb{Q}_{p}\alpha^{r}\beta^{s}\gamma^{t}[x,y]) & n=\textrm{even}\\
\sum_{r+s+t=(n-3)/2}(\mathbb{Q}_{p}\alpha^{r}\beta^{s}\gamma^{t}[x,y,x]+\mathbb{Q}_{p}\alpha^{r}\beta^{s}\gamma^{t}[x,y,y]) & n=\textrm{odd}
\end{cases},\label{eq:sums2}
\end{align}
and we have a similar description for $L^{(n)}$. Let $\mathbb{Q}_{p}\cdot T=\mathbb{Q}_{p}\left\langle t(x),\,t(y),\,t(xy)\right\rangle \subseteq\mathring{\Pi}$.
Like in Propositions \ref{prop:free-1} and \ref{prop:free}, $\mathbb{Q}_{p}\cdot T$
is freely generated by $t(x),t(y),t(xy)$ over $\mathbb{Q}_{p}$,
and hence a unique factorization domain. The following proposition
is an easy corollary of this fact:}
\begin{prop}
\label{prop:alpha}The ring $\mathbb{Q}_{p}\left\langle \alpha,\beta,\gamma\right\rangle \subseteq\mathbb{Q}_{p}\cdot T$
is freely generated by $\alpha,\beta,\gamma$ as an algebra over $\mathbb{Q}_{p}$.
\end{prop}

\begin{cor}
The sums in (\ref{eq:sums2}) are direct. 
\end{cor}

\begin{proof}
We start with showing that the sums in (\ref{eq:sums2}) are direct
when $n$ is even. By Proposition \ref{prop:alpha} it is enough to
show that given an element in $\mathring{L}^{(n)}$ of the form $a=\varepsilon\cdot[x,y]$
for $\varepsilon\in\mathbb{Q}_{p}\left\langle \alpha,\beta,\gamma\right\rangle $,
$\varepsilon$ is determined uniquely. Indeed, consider
\[
\mathbb{Q}_{p}\cdot T\ni t(axy)=t(\varepsilon\cdot[x,y]xy)=\varepsilon\cdot[x,y]^{2}.
\]
Hence, as $\mathbb{Q}_{p}\cdot T$ is a domain, given such $a$ we
can restore $\varepsilon$, as required.

For the odd case, assume that $a=\varepsilon_{1}\cdot[x,y,x]+\varepsilon_{2}\cdot[x,y,y]$
for $\varepsilon_{1},\varepsilon_{2}\in\mathbb{Q}_{p}\left\langle \alpha,\beta,\gamma\right\rangle $.
Observing that
\begin{align*}
\mathbb{Q}_{p}\cdot T\ni t(ax) & =t(\varepsilon_{1}\cdot[x,y,x]x+\varepsilon_{2}\cdot[x,y,y]x)=-2\varepsilon_{2}\cdot[x,y]^{2}\\
\mathbb{Q}_{p}\cdot T\ni t(ay) & =t(\varepsilon_{1}\cdot[x,y,x]y+\varepsilon_{2}\cdot[x,y,y]y)=2\varepsilon_{1}\cdot[x,y]^{2}.
\end{align*}
 we can restore $\varepsilon_{1},\varepsilon_{2}$. Hence, the sums
in (\ref{eq:sums2}) are direct, as required.
\end{proof}
\begin{cor}
\label{cor:formula}We have\foreignlanguage{american}{
\[
\textrm{rank}_{\mathbb{Z}_{p}}(L^{(n)})=\dim_{\mathbb{Q}_{p}}(\mathring{L}^{(n)})=\begin{cases}
n(n+2)/8 & n\,\,\,is\,\,\,even\\
(n-1)(n+1)/4 & n\,\,\,is\,\,\,odd.
\end{cases}
\]
}
\end{cor}

\selectlanguage{american}%
\selectlanguage{english}%
\begin{cor}
The natural map $\mathring{L}_{*}^{(n)}\overset{\sigma}{\to}\mathring{L}^{(n)}$
is an isomorphism of vector spaces for every $n$.
\end{cor}

Considering Proposition \ref{prop:min-1} we obtain:
\begin{cor}
\label{cor:min-new}The ring homomorphism $\sigma$ induces a natural
isomorphism
\[
\hat{G}_{*}=\left\langle \left\langle 1+x_{*},1+y_{*}\right\rangle \right\rangle \cong\hat{G}=\left\langle \left\langle 1+x,1+y\right\rangle \right\rangle 
\]
and for every $g\in\hat{G}$ one has $\min(g)\in\mathring{L}^{(n)}\cap A$
for some $n$.
\end{cor}

Following the outline of the proof of Proposition \ref{prop:min-1},
we can write
\begin{align*}
e^{x} & =1+x+\left(\begin{array}{cc}
x'_{11} & x'_{12}\\
0 & 0
\end{array}\right)\\
e^{y} & =1+y+\left(\begin{array}{cc}
y'_{11} & 0\\
y'_{21} & 0
\end{array}\right)
\end{align*}
where $x_{ij}'$ and $y_{ij}'$ are built up from terms of degree
$>1$. Hence, the ring homomrphism $\phi:\mathring{\Pi}\to\mathring{\Pi}$
defined by sending $x_{ij}\to x_{ij}+x_{ij}'$ and $y_{ij}\to y_{ij}+y_{ij}'$
gives rise to a ring homomorphism $M_{2}(\mathring{\Pi})\to M_{2}(\mathring{\Pi})$
that gives rise to a group homomorphism $\Psi:\hat{G}\to1+M_{2}(\mathring{\Pi},\mathring{Q}_{1})$
defined by $1+x\mapsto e^{x}$ and $1+y\mapsto e^{y}$. Therefore,
as in Subsection \ref{subsec:structure-1}, also here we have
\begin{equation}
\min(g)=\min(\Psi(g))\in\mathring{L}^{(n)}\label{eq:equal}
\end{equation}
for some $n$. We will use this property later (see the proof of Lemma
\ref{prop:technical}).

It is obvious that in general $L^{(n)}\subseteq\mathring{L}^{(n)}\cap A$.
As we stressed in\foreignlanguage{american}{ $\mathsection$\ref{sec:Zubkov}},
Zubkov showed that when $p\neq2$, we actually have $L^{(n)}=\mathring{L}^{(n)}\cap A$
(Proposition \ref{prop:equality}). However, this is not the case
when $p=2$. Here is the full description of $\mathring{L}^{(n)}\cap A$
when $p=2$: 
\begin{prop}
\label{prop:minimal}Denote 
\[
\delta=\frac{\beta^{2}-\alpha\gamma}{4}=t(xy)^{2}-t(xy)t(x)t(y)=[x,y]^{2}
\]
and $S=\mathbb{Z}_{2}\left\langle \alpha,\gamma,\delta\right\rangle $.
Then:
\begin{itemize}
\item If $n$ is even, then $\mathring{L}^{(n)}\cap A=S\cdot[x,y]+\beta\cdot S\cdot[x,y]$.
\item If $n$ is odd, then $\mathring{L}^{(n)}\cap A$ is equal to 
\begin{equation}
S\cdot[x,y,x]+S\cdot[x,y,y]+S\cdot\frac{\beta[x,y,x]+\alpha[x,y,y]}{2}+S\cdot\frac{\gamma[x,y,x]+\beta[x,y,y]}{2}.\label{eq:odd}
\end{equation}
\end{itemize}
\end{prop}

\begin{proof}
One direction of the inclusion, namely $\supseteq$ is easy to verify
(see Part 2 of Lemma \ref{lem:a,b-1} for (\ref{eq:odd})). We turn
to the opposite inclusion, starting with the even case. Let $p(\alpha,\beta,\gamma)\cdot[x,y]\in\mathring{L}^{(n)}\cap A$
when $p(\alpha,\beta,\gamma)$ is a homogeneous polynomial on $\alpha,\beta,\gamma$
over $\mathbb{Q}_{2}$, say of degree $m$. In particular $p(\alpha,\beta,\gamma)\in\mathbb{Q}_{2}\cdot T$.
Consider
\begin{align*}
t(p(\alpha,\beta,\gamma)\cdot[x,y]xy) & =p(\alpha,\beta,\gamma)\cdot t([x,y]xy)=p(\alpha,\beta,\gamma)\cdot[x,y]^{2}.
\end{align*}
On the other hand, $p(\alpha,\beta,\gamma)\cdot[x,y]xy\in A$, so
\[
p(\alpha,\beta,\gamma)\cdot[x,y]^{2}=p(\alpha,\beta,\gamma)\cdot(t(xy)^{2}-t(xy)t(x)t(y))\in T.
\]
One can see that it follows that $p(\alpha,\beta,\gamma)\in T$. So
write
\[
p(\alpha,\beta,\gamma)=\sum_{2i+j+k=2m}\varepsilon_{i,j,k}t(xy)^{i}t(x)^{j}t(y)^{k}
\]
for some $\varepsilon_{i,j,k}\in\mathbb{Z}_{2}$. We want to show
that $p(\alpha,\beta,\gamma)\in S+\beta\cdot S$. Order the triples
$(i,j,k)$ in the lexicographical order and let $(i_{0},j_{0},k_{0})$
be the maximal triple for which $i_{0}$ is even and $\varepsilon_{i_{0},j_{0},k_{0}}\neq0$.
We have two options:
\begin{itemize}
\item If $j_{0},k_{0}$ are even, reduce $\varepsilon_{i_{0},j_{0},k_{0}}\cdot\delta^{\frac{i_{0}}{2}}\cdot\alpha^{\frac{j_{0}}{2}}\cdot\gamma^{\frac{k_{0}}{2}}\in S$
from $p(\alpha,\beta,\gamma)$. One can see that by doing this, the
maximal triple $(i_{0},j_{0},k_{0})$ for which $i_{0}$ is even,
and $\varepsilon_{i_{0},j_{0},k_{0}}\neq0$ is reduced.
\item If $j_{0},k_{0}$ are odd, noticing that $(i_{0}+1)^{-1}\in\mathbb{Z}_{2}$,
consider 
\begin{align*}
 & (i_{0}+1)^{-1}\cdot\varepsilon_{i_{0},j_{0},k_{0}}\cdot\beta\cdot\delta^{\frac{i_{0}}{2}}\cdot\alpha^{\frac{j_{0}-1}{2}}\cdot\gamma^{\frac{k_{0}-1}{2}}\\
 & =(i_{0}+1)^{-1}\cdot2\cdot\varepsilon_{i_{0},j_{0},k_{0}}\cdot t(xy)^{i_{0}+1}\cdot\alpha^{\frac{j_{0}-1}{2}}\cdot\gamma^{\frac{k_{0}-1}{2}}\\
 & \,\,\,\,\,\,\,\,-\varepsilon_{i_{0},j_{0},k_{0}}t(xy)^{i_{0}}\cdot t(x)\cdot\alpha^{\frac{j_{0}-1}{2}}\cdot t(y)\gamma^{\frac{k_{0}-1}{2}}\\
 & \,\,\,\,\,\,\,\,+\textrm{terms\,\,with\,\,lower\,\,powers\,\,of}\,\,t(xy).
\end{align*}
Hence, by adding $(i_{0}+1)^{-1}\cdot\varepsilon_{i_{0},j_{0},k_{0}}\cdot\beta\cdot\delta^{\frac{i_{0}}{2}}\cdot\alpha^{\frac{j_{0}-1}{2}}\cdot\gamma^{\frac{k_{0}-1}{2}}\in\beta\cdot S$
to $p(\alpha,\beta,\gamma)$, the maximal triple $(i_{0},j_{0},k_{0})$
for which $i_{0}$ is even and $\varepsilon_{i_{0},j_{0},k_{0}}\neq0$
is reduced. 
\end{itemize}
We continue in this way until all the terms for which $i_{0}$ is
even are reduced. We claim that in this case $p(\alpha,\beta,\gamma)=0$.
Indeed, write
\[
p(\alpha,\beta,\gamma)=\sum_{i=0}^{r}\beta^{i}\cdot q_{i}(\alpha,\gamma).
\]
Clearly, the highest $i_{0}$ for which $q_{i_{0}}(\alpha,\gamma)\neq0$
is the highest degree of $t(xy)$. So by assumption, $i_{0}$ is odd.
Hence, the term for which the degree of $t(xy)$ is $i_{0}-1$ is
given by
\[
2^{i_{0}-1}\cdot t(xy)^{i_{0}-1}(-i_{0}\cdot t(x)t(y)q_{i_{0}}(\alpha,\gamma)+q_{i_{0}-1}(\alpha,\gamma)).
\]
As $q_{i_{0}}(\alpha,\gamma)\neq0$, the degrees of the terms of $t(x)t(y)q_{i_{0}}(\alpha,\gamma)$
in $t(x),t(y)$ are odd. Thus, as the degrees of the terms of $q_{i_{0}-1}(\alpha,\gamma)$
in $t(x),t(y)$ are even, one has
\[
-i_{0}\cdot t(x)t(y)q_{i_{0}}(\alpha,\gamma)+q_{i_{0}-1}(\alpha,\gamma)\neq0,
\]
and this is a contradiction to the assumption that $p(\alpha,\beta,\gamma)$
does not have terms with even degree in $t(xy)$.

For the odd case, write
\[
v=p(\alpha,\beta,\gamma)\cdot[x,y,x]+q(\alpha,\beta,\gamma)\cdot[x,y,y]\in\mathring{L}^{(n)}\cap A
\]
where $p(\alpha,\beta,\gamma),\,q(\alpha,\beta,\gamma)\in\mathbb{Q}_{2}\cdot T$
are homogeneous in $\alpha,\beta,\gamma$ of degree $m$. Using the
identities $[x,y,x]=-t(x)[x,y]+2[x,y]x$, $[x,y,y]=-t(y)[x,y]+2[x,y]y$,
we have
\[
v=(-p(\alpha,\beta,\gamma)t(x)-q(\alpha,\beta,\gamma)t(y))[x,y]+2p(\alpha,\beta,\gamma)[x,y]x+2q(\alpha,\beta,\gamma)[x,y]y.
\]
As $v\in A$ we have 
\[
-t(xy)[x,y]^{2}(p(\alpha,\beta,\gamma)t(x)+q(\alpha,\beta,\gamma)t(y))=t(x\cdot v\cdot[x,y]y)\in T.
\]
Also here, it follows that $p(\alpha,\beta,\gamma)t(x)+q(\alpha,\beta,\gamma)t(y)\in T$.
So write
\[
p(\alpha,\beta,\gamma)t(x)+q(\alpha,\beta,\gamma)t(y)=\sum_{2i+j+k=2m+1}\varepsilon_{i,j,k}t(xy)^{i}t(x)^{j}t(y)^{k}
\]
for some $\varepsilon_{i,j,k}\in\mathbb{Z}_{2}$. Order the triples
$(i,j,k)$ in the lexicographical order and let $(i_{0},j_{0},k_{0})$
be the maximal triple for which $\varepsilon_{i_{0},j_{0},k_{0}}\neq0$.
We have a few options:
\begin{itemize}
\item If $i_{0}$ is even, $j_{0}$ is odd, $k_{0}$ is even, subtract the
following from $v$:
\[
\varepsilon_{i_{0},j_{0},k_{0}}\cdot\delta^{\frac{i_{0}}{2}}\cdot\alpha^{\frac{j_{0}-1}{2}}\cdot\gamma^{\frac{k_{0}}{2}}[x,y,x]\in S\cdot[x,y,x].
\]
\item If $i_{0}$ is even, $j_{0}$ is even, $k_{0}$ is odd, subtract the
following from $v$:
\[
\varepsilon_{i_{0},j_{0},k_{0}}\cdot\delta^{\frac{i_{0}}{2}}\cdot\alpha^{\frac{j_{0}}{2}}\cdot\gamma^{\frac{k_{0}-1}{2}}[x,y,y]\in S\cdot[x,y,y].
\]
\item If $i_{0}$ is odd, $j_{0}$ is odd, $k_{0}$ is even, subtract the
following from $v$:
\[
\varepsilon_{i_{0},j_{0},k_{0}}\cdot\delta^{\frac{i_{0}-1}{2}}\cdot\alpha^{\frac{j_{0}-1}{2}}\cdot\gamma^{\frac{k_{0}}{2}}\cdot\left(\frac{\beta[x,y,x]+\alpha[x,y,y]}{2}\right)\in S\cdot\frac{\beta[x,y,x]+\alpha[x,y,y]}{2}.
\]
\item If $i_{0}$ is odd, $j_{0}$ is even, $k_{0}$ is odd, subtract the
following from $v$:
\[
\varepsilon_{i_{0},j_{0},k_{0}}\cdot\delta^{\frac{i_{0}-1}{2}}\cdot\alpha^{\frac{j_{0}}{2}}\cdot\gamma^{\frac{k_{0}-1}{2}}\left(\frac{\gamma[x,y,x]+\beta[x,y,y]}{2}\right)\in S\cdot\frac{\gamma[x,y,x]+\beta[x,y,y]}{2}.
\]
\end{itemize}
One can see that in the new element, the maximal $(i_{0},j_{0},k_{0})$
for which $\varepsilon_{i_{0},j_{0},k_{0}}\neq0$ is lower. We continue
the process until the expression 
\begin{equation}
p(\alpha,\beta,\gamma)t(x)+q(\alpha,\beta,\gamma)t(y)\label{eq:expression}
\end{equation}
vanishes. We claim that in this case, $p(\alpha,\beta,\gamma)=q(\alpha,\beta,\gamma)=0$,
i.e. $v=0$. Indeed, if we write $p(\alpha,\beta,\gamma)$ as an expression
in $t(x,y),t(x),t(y)$ (with coefficents in $\mathbb{Q}_{2})$ and
order the monomials with the above lexicographical order, the highest
monomial of $p(\alpha,\beta,\gamma)t(x)$ will have the form $\varepsilon_{1}\cdot t(xy)^{i}t(x)^{j}t(y)^{k}$
where $j$ is odd and $k$ is even, and the highest monomial of $q(\alpha,\beta,\gamma)t(y)$
will have the form $\varepsilon_{2}\cdot t(xy)^{i}t(x)^{j}t(y)^{k}$
where $j$ is even and $k$ is odd. So they cannot cancel each other,
and the only way for (\ref{eq:expression}) to be $0$ is that $p(\alpha,\beta,\gamma)=q(\alpha,\beta,\gamma)=0$,
as required.
\end{proof}

\subsection{Proving Proposition \ref{prop:dichotomy}}

We want now to prove Proposition \ref{prop:dichotomy}. Recalling
the ideals $Q_{n}\vartriangleleft\Pi$ we denote
\begin{align*}
\omega_{n}(\hat{G}) & =\ker(\hat{G}\to GL_{2}(\Pi/Q_{n}))\\
\Omega_{n}(\hat{G}) & =\omega_{n}(\hat{G})/\omega_{n+1}(\hat{G}).
\end{align*}

\selectlanguage{american}%
Just like in $\mathsection$\ref{sec:Zubkov}, by Corollary \ref{cor:min-new},
one can view $\Omega_{n}(\hat{G})$ as an abelian subgroup of $\mathring{L}^{(n)}\cap A$
such that
\begin{equation}
\Upsilon_{n}(\hat{G})\twoheadrightarrow L^{(n)}\leq\Omega_{n}(\hat{G})\leq\mathring{L}^{(n)}\cap A.\label{eq:inequality}
\end{equation}

\selectlanguage{english}%
Let's start with proving the first part of Proposition \ref{prop:dichotomy},
namely, that for every $p$ and every $n\leq5$ one has $\Upsilon_{n}(\hat{G})\cong\mathbb{Z}_{p}^{l_{2}(n)}$.
Actually:
\selectlanguage{american}%
\begin{prop}
\label{prop:part 1}For every $p$ we have:
\begin{itemize}
\item $\hat{G}_{n}=\omega_{n}(\hat{G})$ for every $n\leq6$.
\item $\Upsilon_{n}(\hat{F})\cong\Upsilon_{n}(\hat{G})\cong\Omega_{n}(\hat{G})\cong L^{(n)}\cong\mathbb{Z}_{p}^{l_{2}(n)}$
for every $n\leq5$.
\end{itemize}
\end{prop}

\begin{proof}
By definition, we have $\hat{G}=\hat{G}_{1}=\omega_{1}(\hat{G})$.
We want to show, by induction on $n$, that the same is valid for
every $n\leq6$. Let $n\leq5$, and assume that $\hat{G}_{n}=\omega_{n}(\hat{G})$.
Under this assumption, $\Upsilon_{n}(\hat{G})\twoheadrightarrow\Omega_{n}(\hat{G})$
and hence equation (\ref{eq:inequality}) gives
\[
\Upsilon_{n}(\hat{F})\twoheadrightarrow\Upsilon_{n}(\hat{G})\twoheadrightarrow L^{(n)}=\Omega_{n}(\hat{G}).
\]
Hence, as for every $n\leq5$ the formula in Corollary \ref{cor:formula}
coincides with the Witt formula, we have
\[
\Upsilon_{n}(\hat{F})\cong\mathbb{Z}^{l_{2}(n)},\,\,\,\,\,\,\Omega_{n}(\hat{G})\cong L^{(n)}\cong\mathbb{Z}^{l_{2}(n)}.
\]
As $\mathbb{Z}^{l_{2}(n)}$ is Hopfian (as a finitely generated profinite
group), we get that the composition map $\Upsilon_{n}(\hat{F})\twoheadrightarrow\Omega_{n}(\hat{G})$
is an isomorphism. It follows that 
\[
\Upsilon_{n}(\hat{F})\twoheadrightarrow\Upsilon_{n}(\hat{G})=\hat{G}_{n}/\hat{G}_{n+1}\twoheadrightarrow\Omega_{n}(\hat{G})=\omega_{n}(\hat{G})/\omega_{n+1}(\hat{G})
\]
is an isomorphism. In particular $\hat{G}_{n+1}=\omega_{n+1}(\hat{G})$
as required. Notice that in the course of the proof we also proved
the second statement.
\end{proof}
As mentioned in $\mathsection$\ref{sec:Zubkov}, when $p\neq2$ we
have $L^{(n)}=\mathring{L}^{(n)}\cap A$, and hence equation (\ref{eq:inequality})
actually gives
\[
\Upsilon_{n}(\hat{G})\twoheadrightarrow\Omega_{n}(\hat{G})\cong L^{(n)}.
\]
Thus, in this case, using the idea in $\mathsection$\ref{sec:Zubkov}
that was used in order to construct a pro-$p$ identity, one can easily
deduce the following proposition, yielding the second part of Proposition
\ref{prop:dichotomy} (see \cite{key-3}, Theorem 4.1):
\begin{prop}
\label{prop:deg=00003Dlower-central}When $p\neq2$ we have $\hat{G}_{n}=\omega_{n}(\hat{G})$
and hence $\Upsilon_{n}(\hat{G})=\Omega_{n}(\hat{G})\cong L^{(n)}$
for every $n$.
\end{prop}

\selectlanguage{english}%
The meaning of Proposition \ref{prop:deg=00003Dlower-central} is
that when $p\neq2$, given an element $g\in\hat{G}$, there is a correspondence
between the degree of $\min(g)$ and the location of $g$ in lower
central series of $\hat{G}$. When $p=2$, this correspondence is
broken, as we demonstrate below. 

Assume now that $p=2$. Recall that \foreignlanguage{american}{$\hat{G}_{7}\subseteq\omega_{7}(\hat{G})$
and} by Proposition \ref{prop:part 1} we have \foreignlanguage{american}{$\hat{G}_{6}=\omega_{6}(\hat{G})$.
It follows that we have the exact sequence
\begin{align*}
1\to(\omega_{7}(\hat{G})\cap\hat{G}_{6})/\hat{G}_{7} & \hookrightarrow\Upsilon_{6}(\hat{G})=\hat{G}_{6}/\hat{G}_{7}\twoheadrightarrow\Omega_{6}(\hat{G})=\omega_{6}(\hat{G})/\omega_{7}(\hat{G})\to1.
\end{align*}
Notice that by equation (\ref{eq:inequality}), the surjective map
$\Upsilon_{6}(\hat{G})\twoheadrightarrow\Omega_{6}(\hat{G})$ yields
that $\Omega_{6}(\hat{G})\cong L^{(6)}\cong\mathbb{Z}_{2}^{6}$.}

\selectlanguage{american}%
Given a pro-$p$ group $\hat{H}$, we denote the minimum number of
(topological) generators for $\hat{H}$ by $d(\hat{H})$. We have
the following lemma:
\begin{lem}
Let $1\to\hat{H}_{1}\to\hat{H}_{2}\to\hat{H}_{3}\to1$ be an exact
squence of finitely generated abelian pro-$p$ groups, such that $\hat{H}_{3}$
is free (as an abelian pro-$p$ group). Then: $d(\hat{H}_{2})=d(\hat{H}_{1})+d(\hat{H}_{3})$.
\end{lem}

\begin{proof}
As $\hat{H}_{3}$ is free, we have a section map $\hat{H}_{3}\to\hat{H}_{2}$
such that the composition map $\hat{H}_{3}\to\hat{H}_{2}\to\hat{H}_{3}$
is a natural isomorphism. It follows that: $\hat{H}_{2}\cong\hat{H}_{1}\rtimes\hat{H}_{3}=\hat{H}_{1}\times\hat{H}_{3}$.
Hence $d(\hat{H}_{2})\leq d(\hat{H}_{1})+d(\hat{H}_{3})$. On the
other hand, $\hat{H}_{2}$ has the vector space
\[
(\hat{H}_{1}/\Phi(\hat{H}_{1}))\times(\hat{H}_{3}/\Phi(\hat{H}_{3}))\cong(\mathbb{Z}/p\mathbb{Z})^{d(\hat{H}_{1})}\times(\mathbb{Z}/p\mathbb{Z})^{d(\hat{H}_{3})}\cong(\mathbb{Z}/p\mathbb{Z})^{d(\hat{H}_{1})+d(\hat{H}_{3})}
\]
as a homomorphic image, where $\Phi(\hat{H}_{1}),\Phi(\hat{H}_{3})$
are the Frattini subgroups of $\hat{H}_{1},\hat{H}_{3}$. Hence, $d(\hat{H}_{2})\geq d(\hat{H}_{1})+d(\hat{H}_{3})$,
as required.
\end{proof}
Recall that $\Omega_{6}(\hat{G})\cong\mathbb{Z}_{2}^{6}$. By the
lemma, it follows that in order to prove the last part of Proposition
\ref{prop:dichotomy}, it is enough to show that for 
\[
\hat{H}=(\omega_{7}(\hat{G})\cap\hat{G}_{6})/\hat{G}_{7}
\]
we have $d(\hat{H})\geq3$. By Corollary \ref{cor:min-new} we have
a natural map $\chi:\omega_{7}(\hat{G})\to\mathring{L}^{(7)}$ defined
by sending each $g\in\omega_{7}(\hat{G})$ to its corresponding minimal
term\footnote{Elements $g\in\omega_{7}(\hat{G})$ with $\deg(\min(g))>7$ are mapped
to $0\in\mathring{L}^{(7)}$.}. Clearly, the image of $\hat{G}_{7}$ under $\chi$ is $L^{(7)}$.
Hence 
\[
\overline{H}=\chi(\omega_{7}(\hat{G})\cap\hat{G}_{6})/\chi(\hat{G}_{7})=\chi(\omega_{7}(\hat{G})\cap\hat{G}_{6})/L^{(7)}
\]
is a homomorphic image of $\hat{H}$. We are going to show that the
abelian group $\overline{H}$ contains a copy of $(\mathbb{Z}/2\mathbb{Z})^{3}$.
It will follow that $d(\overline{H})\geq3$, and thus $d(\hat{H})\geq3$
as well, as required. We prove the following technical proposition:
\begin{lem}
\label{prop:technical}The following elements belong to $\omega_{7}(\hat{G})\cap\hat{G}_{6}$:
\begin{align*}
g_{1} & =[1+x,1+y,1+x,1+x,[1+x,1+y]_{\hat{G}}]_{\hat{G}}\\
g_{2} & =[1+x,1+y,1+x,1+y,[1+x,1+y]_{\hat{G}}]_{\hat{G}}\\
g_{3} & =[1+x,1+y,1+y,1+y,[1+x,1+y]_{\hat{G}}]_{\hat{G}}.
\end{align*}
In addition, the image of $g_{i}$ under $\chi$ is
\begin{align*}
\chi(g_{1})=\min(g_{1})= & \frac{\beta^{2}-\alpha\gamma}{2}[x,y,x]=2\delta\cdot[x,y,x]\,\,\,\,\textrm{mod }L^{(7)}\\
\chi(g_{2})=\min(g_{2})= & \beta\cdot(\frac{\beta[x,y,x]+\alpha[x,y,y]}{2})\\
 & +\beta\cdot(\frac{\beta[x,y,y]+\gamma[x,y,x]}{2})\,\,\,\,\textrm{mod }L^{(7)}\\
\chi(g_{3})=\min(g_{3})= & \frac{\beta^{2}-\alpha\gamma}{2}[x,y,y]=2\delta\cdot[x,y,y]\,\,\,\,\textrm{mod }L^{(7)}
\end{align*}
\end{lem}

Before we prove this lemma, we want to show that it yields that $\overline{H}$
contains a copy of $(\mathbb{Z}/2\mathbb{Z})^{3}$. Indeed, we saw
that the elements
\begin{align*}
\alpha^{r}\beta^{s}\gamma^{t}[x,y,x] & ,\,\,\,\,r+s+t=2\\
\alpha^{r}\beta^{s}\gamma^{t}[x,y,y] & ,\,\,\,\,r+s+t=2
\end{align*}
generate the $\mathbb{Z}_{2}$-module $L^{(7)}$. In addition, as
these elements form a basis of the vector space $\mathring{L}^{(7)}$,
they generate $L^{(7)}$ \uline{freely} as a $\mathbb{Z}_{2}$-module.
It follows that as $\frac{1}{2}\notin\mathbb{Z}_{2}$, the elements
$\chi(g_{i})$ are not in $L^{(7)}$. Moreover, obviously the order
of $\chi(g_{i})$ mod $L^{(7)}$ is $2$. Eventually, one can easily
see that the $\chi(g_{i})$ are different mod $L^{(7)}$, and none
of these elements can be expressed in terms of the others. It follows
that the subgroup of $\overline{H}$ generated by $\chi(g_{i})$ mod
$L^{(7)}$ is isomorphic to $(\mathbb{Z}/2\mathbb{Z})^{3}\leq\overline{H}$,
as required. So it remains to prove Lemma \ref{prop:technical}. 
\begin{proof}
(of Lemma \ref{prop:technical}) By the formula given in equation
(\ref{eq:equal}), in order to evaluate $\min(g_{1})$ it is enough
to evaluate 
\[
\min([e^{x},e^{y},e^{x},e^{x},[e^{x},e^{y}]_{\Psi(\hat{G})}]_{\Psi(\hat{G})})
\]
\foreignlanguage{english}{So in general, if $z,w\in M_{2}(\mathring{\Pi}/\mathring{Q}_{1})$
then a direct computation using the Baker-Campbell-Hausdorff formula
gives
\begin{align*}
\ln([e^{z},e^{w}]_{\Psi(\hat{G})})= & [z,w]-\frac{1}{2}[z,w,z]-\frac{1}{2}[z,w,w]+\textrm{higher\,\,commutators\,\,on}\,\,z,w.
\end{align*}
Using this formula three times one has
\begin{align*}
\ln([e^{x},e^{y},e^{x},e^{x}]_{\Psi(\hat{G})})= & \alpha[x,y]-\frac{3\alpha}{2}[x,y,x]-\frac{\beta}{2}[x,y,x]+\textrm{terms\,\,of\,\,degree}\geq6.
\end{align*}
Therefore, using the identities
\begin{align*}
[x,y,[x,y,x]]= & [[x,y,x],y,x]-[[x,y,x],x,y]=\beta[x,y,x]-\alpha[x,y,y]\\{}
[x,y,[x,y,y]]= & [[x,y,y],y,x]-[[x,y,y],x,y]=\gamma[x,y,x]-\beta[x,y,y]
\end{align*}
one has
\begin{align*}
 & \ln([e^{x},e^{y},e^{x},e^{x},[e^{x},e^{y}]_{\Psi(\hat{G})}]_{\Psi(\hat{G})})\\
 & =[\alpha[x,y]-\frac{3\alpha}{2}[x,y,x]-\frac{\beta}{2}[x,y,x],[x,y]-\frac{1}{2}[x,y,x]-\frac{1}{2}[x,y,y]]\\
 & \,\,\,\,\,\,\,\,+\textrm{terms\,\,of\,\,degree}\geq8\\
 & =\frac{\beta^{2}-\alpha\gamma}{2}\cdot[x,y,x]+\alpha\beta[x,y,x]-\alpha^{2}[x,y,y]+\textrm{terms\,\,of\,\,degree}\geq8.
\end{align*}
It follows that 
\[
\min(g_{1})=\min([e^{x},e^{y},e^{x},e^{x},[e^{x},e^{y}]_{\Psi(\hat{G})}]_{\Psi(\hat{G})})=\frac{\beta^{2}-\alpha\gamma}{2}\cdot[x,y,x]\,\,\,\,\textrm{mod }L^{(7)}.
\]
A similar computation shows that 
\[
\min(g_{3})=\min([e^{x},e^{y},e^{y},e^{y},[e^{x},e^{y}]_{\Psi(\hat{G})}]_{\Psi(\hat{G})})=\frac{\beta^{2}-\alpha\gamma}{2}\cdot[x,y,y]\,\,\,\,\textrm{mod }L^{(7)}.
\]
Regarding $g_{2}$ we have
\begin{align*}
\ln([e^{x},e^{y},e^{x},e^{y}]_{\Psi(G)})= & \beta[x,y]-\alpha[x,y,y]-\beta[x,y,y]+\textrm{terms\,\,of\,\,degree}\geq6.
\end{align*}
Therefore
\begin{align*}
 & \ln([e^{x},e^{y},e^{x},e^{y},[e^{x},e^{y}]_{\Psi(\hat{G})}]_{\Psi(\hat{G})})\\
 & =[\beta[x,y]-\alpha[x,y,y]-\beta[x,y,y],[x,y]-\frac{1}{2}[x,y,x]-\frac{1}{2}[x,y,y]]\\
 & \,\,\,\,\,\,\,\,+\textrm{terms\,\,of\,\,degree}\geq8\\
 & =-\frac{\beta^{2}}{2}[x,y,x]-\frac{\alpha\beta}{2}[x,y,y]+\frac{\beta\gamma}{2}[x,y,x]-\frac{\beta^{2}}{2}[x,y,y]+\alpha\gamma[x,y,x]\\
 & \,\,\,\,\,\,\,\,+\textrm{terms\,\,of\,\,degree}\geq8.
\end{align*}
Hence
\begin{align*}
\min(g_{2}) & =\min([e^{x},e^{y},e^{x},e^{y},[e^{x},e^{y}]_{\Psi(\hat{G})}]_{\Psi(\hat{G})})\\
 & =\beta\cdot(\frac{\beta[x,y,x]+\alpha[x,y,y]}{2})+\beta\cdot(\frac{\beta[x,y,y]+\gamma[x,y,x]}{2})\,\,\,\,\textrm{mod }L^{(7)}
\end{align*}
as required.}
\end{proof}
\selectlanguage{english}%

Ben-Ezra, David El-Chai, Department of Mathematics, University of
California in San-Diego, La Jolla, CA 92093, USA (dbenezra@mail.huji.ac.il)\\
\\
Zelmanov, Efim, Department of Mathematics, University of California
in San-Diego, La Jolla, CA 92093, USA (ezelmanov@ucsd.edu)
\end{document}